\documentclass{amsart}

\usepackage{amsmath,amsthm,amssymb,amsfonts,bbm,amscd,mathrsfs}
\usepackage{tikz,tikz-cd}
\usepackage{url,graphicx}
\usepackage[T1]{fontenc}
\usepackage[pagebackref,colorlinks]{hyperref}
\usepackage{enumitem}

\newtheorem{thm}{Theorem}[section]
\newtheorem{lem}[thm]{Lemma}
\newtheorem*{thmx}{Main Theorem}
\newtheorem{prop}[thm]{Proposition}
\newtheorem{cor}[thm]{Corollary}

\numberwithin{equation}{section}

\theoremstyle{definition}
\newtheorem{rmk}[thm]{Remark}

\newtheorem*{keyob}{Expectation}
\newtheorem{defn}[thm]{Definition}

\newtheorem{claim}[thm]{Claim}

\title{Proper good quotients for $\mathbf{G}_m$-actions}

\author{Xucheng Zhang}
\date{\today}
\address{Yau Mathematical Sciences Center, Tsinghua University, Beijing 100084, China \\
Email: \href{mailto: xucheng.zhang@stud.uni-due.de}{zhangxucheng@mail.tsinghua.edu.cn}}

\begin{document}

\begin{abstract}
We give an algebraic proof of a result, due to Bia\l{}ynicki-Birula and Sommese, characterizing the invariant open subsets of a normal proper variety equipped with a $\mathbf{G}_m$-action that admit a proper good quotient. A major ingredient is the existence result for moduli spaces of algebraic stacks due to Alper, Halpern-Leistner and Heinloth.
\end{abstract}

\maketitle

\section{Introduction}

For algebraic stacks arising from moduli problems, the existence criterion of moduli spaces \cite{MR4665776} gives a method to find out those open substacks that admit proper moduli spaces. In this article we consider the problem in the case of global quotient stacks $[X/\mathbf{G}_m]$ where $X$ is a normal proper variety equipped with a $\mathbf{G}_m$-action. In other words, we want to find those $\mathbf{G}_m$-invariant open subsets of $X$ that admit a proper good quotient. For a meromorphic locally linearizable $\mathbf{G}_m$-action on a normal compact analytic space $X$, Bia\l{}ynicki-Birula and Sommese \cite{MR709583} provided a combinatorial description of those open subsets with compact geometric quotients (see also \cite{MR755486} for the result on compact good quotients).

To formulate the result let $\mathbf{G}_m$ act on a normal proper variety $X$ and let $X^{0} \subseteq X$ be the closed subscheme of fixed points. One special thing about $\mathbf{G}_m$-actions is that they form a canonical flow on $X$. The flow-lines define a natural relation on the set $\pi_0(X^{0})$ of connected components of $X^{0}$. 

Building on these datum Bia\l{}ynicki-Birula and Sommese introduced the notion of sections, a division $(A^-,A^+)$ of the set $\pi_0(X^{0})$ with respect to this relation. Any section defines an open subset of $X$ consisting of points that flow from a fixed point component in $A^-$ to another fixed point component in $A^+$. The authors then proved that all open subsets with proper geometric quotients arise in this way. They also formulated a variant, called semi-sections, to describe open subsets with proper good (but not necessarily geometric) quotients.

The main result of this article is an algebraic version of their result, which works over arbitrary fields.
\begin{thmx}[Theorem \ref{onlyif}, Theorem \ref{prop:if}]
Let $X$ be a normal proper variety over a field $k$ with a $\mathbf{G}_m$-action. There is a one-to-one correspondence
\[
\begin{Bmatrix} 
\text{Semi-sections on } X \\
\text{(Definition \ref{def:ss})}
\end{Bmatrix}
\leftrightarrow
\begin{Bmatrix} 
\mathbf{G}_m\text{-invariant open dense subsets of } X \\
\text{with proper good quotients}
\end{Bmatrix},
\]
given by assigning each semi-section to the corresponding semi-sectional subset. 

Moreover, under this correspondence $\mathbf{G}_m$-invariant open dense subsets of $X$ with proper geometric quotients correspond to sections on $X$.
\end{thmx}
\begin{rmk}
As a corollary, there are only finitely many open subsets of $X$ with proper good quotients, and $X$ is covered by such open subsets if in addition $X$ is projective (Proposition \ref{prop-app}). The latter can be compared with the counterpart on the moduli stack of rank 2 vector bundles over a curve (see \cite[Theorem 1]{zhang-rank-2} or \cite[Theorem B]{DZ2023}), where only simple or semistable vector bundles admit open neighbourhoods with good moduli spaces.
\end{rmk}
The bijection in \textbf{Main Theorem} is established using the result of \cite{MR4665776} that characterizes algebraic stacks with proper moduli spaces, which involves three local criteria: $\Theta$-reductivity ($\Theta$), S-completeness (S) and the existence part of valuative criterion for properness (E). The $\Theta$-reductivity of general quotient stacks is characterized in \cite[Proposition 3.13]{MR4665776} and we adapt it in our situation (Lemma \ref{theta-red}). For the last two conditions by definition we have to consider degenerations of \emph{all} orbit closures and it could be a mess. 

Fortunately, we find that for certain stacks (including quotient stacks from $\mathbf{G}_m$-actions) it suffices to consider degenerations of the \emph{generic} orbit closure, which simplifies a lot. For (E) this actually holds in a great generality \cite[\href{https://stacks.math.columbia.edu/tag/0CQM}{Tag 0CQM}]{stacks-project}. But the situation for (S) is more subtle and this comes our first key input. We show (Theorem \ref{thm:S-com-new}) that S-completeness can be checked on families that cover the generic point, generalizing the refined valuative criterion for separatedness of morphisms between algebraic spaces (\cite[\href{https://stacks.math.columbia.edu/tag/0CME}{Tag 0CME}]{stacks-project}). This seems to be new and might be of independent interest, as S-completeness is the most pivotal element in the construction of separated moduli spaces. The proof is somehow technical and relies on a detailed study of some rather complicated mapping stacks.

Therefore degenerations of the generic orbit can serve as test objects for both (S) and (E) and we model them into a definition (Definition \ref{key-defn:sm}). In addition to their simple form, the effectiveness of these individual objects saves us from constructing the Douady space $Q$ that parameterizes all degenerations of the generic orbit, and deriving its numerous properties, as in \cite[(0.1.2) Theorem]{MR709583}. This is our second key input. In addition, it enables us to avoid some technical difficulties in the original arguments.

Then we manage to formulate a geometric criterion when quotient stacks from $\mathbf{G}_m$-actions satisfy (S) and (E) in terms of the test objects (Theorem \ref{summary}). This criterion, however, does not seem powerful enough to deduce a description of those open subsets with proper good quotients, even coupled with ($\Theta$) (Lemma \ref{theta-red}), it does hint a topological criterion when a separated good quotient is proper (Proposition \ref{prop:2-connected}), which among other things helps to formulate the notion of semi-sections (Definition \ref{def:ss}). The key step is to compare weights on certain cohomology groups, and in the singular case we use the intersection cohomology. This is our third key input. With these two criteria, we are able to describe those open subsets with proper good quotients (Theorem \ref{onlyif}). 

Several arguments in this article are similar to those in \cite{MR709583} or \cite{MR755486}, but the new formulation in terms of existence criteria \cite{MR4665776} makes the framework very transparent and coherent, and also put us in a good place to state and prove intermediate results under less assumptions. In particular we will see that the mysterious notion of semi-sections appears naturally just from the condition imposed by S-completeness.

\subsection*{Acknowledgement}
The work presented in this article constitutes the second half of my Ph.D. thesis carried out at Universit\"{a}t Duisburg-Essen. I would like to thank my supervisor Jochen Heinloth for suggesting this topic and sharing many ideas along this project. I would also like to thank Yifei Chen, Baohua Fu and Daniel Greb for discussions on the smoothability of maximal chain of orbits.
\section{Preliminaries}\label{Preliminary}
\subsection{A refined valuative criterion for S-completeness}
For any DVR $R$ with fraction field $K$, residue field $\kappa$, and uniformizer $\pi \in R$, as in \cite[\S 2.B]{MR3758902}, the ``separatedness test space'' is defined as the quotient stack
\[
\overline{\mathrm{ST}}_R:=[\mathrm{Spec}(R[x,y]/xy-\pi)/\mathbf{G}_m],
\]
where $x,y$ have $\mathbf{G}_m$-weights $1,-1$ respectively. Denote by $0:=\mathrm{B}\mathbf{G}_{m,\kappa}=[\mathrm{Spec}(\kappa)/\mathbf{G}_m]$ its unique closed point defined by the vanishing of both $x$ and $y$.
\begin{defn}[\cite{MR4665776}, Definition 3.38]
A morphism $f: \mathscr{X} \to \mathscr{Y}$ of locally noetherian algebraic stacks is \emph{S-complete} if for every DVR $R$, any commutative diagram
\begin{equation}\label{D1}
\begin{tikzcd}
\overline{\mathrm{ST}}_R -\{0\} \ar[r] \ar[d,hook] & \mathscr{X} \ar[d,"f"] \\
\overline{\mathrm{ST}}_R \ar[r] \ar[ur,dashed,"\exists !"'] & \mathscr{Y}
\end{tikzcd}
\end{equation}
of solid arrows can be uniquely filled in.
\end{defn}
In this section we prove that in some cases it suffices to check S-completeness for families that cover the generic point. To be precise
\begin{thm}\label{thm:S-com-new}
Let $f: \mathscr{X} \to \mathscr{Y}$ be a quasi-compact morphism of algebraic stacks, locally of finite type and with affine diagonal over a locally excellent quasi-separated algebraic space $S$. Assume that $h: \mathscr{X}^\circ \to \mathscr{X}$ is a dominant morphism of algebraic stacks, finite type over $S$. Suppose one of the following holds:
\begin{enumerate}
\item
$\mathscr{X}$ is locally reductive and admits an adequate moduli space, $\mathscr{Y}=S$ and $f$ is the structure morphism.
\item
$\mathscr{X}$ admits a cover by separated representable \'{e}tale morphisms $[X/\mathbf{G}_m] \to \mathscr{X}$, where $X$ is a quasi-separated quasi-compact algebraic space.
\end{enumerate}
Then $f$ is S-complete if and only if $f$ is S-complete relative to $h$, i.e., for every DVR $R$ with fraction field $K$, any commutative diagram
\[
\begin{tikzcd}
\overline{\mathrm{ST}}_R-\{0\} \ar[r] \ar[d,hook] & \mathscr{X} \ar[d,"f"] \\
\overline{\mathrm{ST}}_R \ar[ur,dashed,"\exists!"'] \ar[r] & \mathscr{Y}
\end{tikzcd}
\]
of solid arrows such that $\mathrm{Spec}(K) \hookrightarrow \overline{\mathrm{ST}}_R-\{0\} \to \mathscr{X}$ factors through $h: \mathscr{X}^\circ \to \mathscr{X}$ can be uniquely filled in.
\end{thm}
Since a morphism between quasi-separated, locally noetherian algebraic spaces is S-complete if and only if it is separated (see \cite[Proposition 3.46]{MR4665776}), this generalizes the refined valuative criterion for separatedness of morphisms between algebraic spaces (see, e.g. \cite[\href{https://stacks.math.columbia.edu/tag/0CME}{Tag 0CME}]{stacks-project}). For algebraic stacks the situation is more subtle and we need to study some rather complicated mapping stacks.
\subsubsection{Case (1)}
Suppose $\mathscr{X} \to S$ is S-complete relative to $h$. Let $\mathscr{X} \to X$ be the adequate moduli space. Let $X^\circ \subseteq X$ be an open dense algebraic subspace contained in the image of the dominant morphism $\mathscr{X}^\circ \to \mathscr{X} \to X$. Then $\mathscr{X} \to S$ is S-complete if and only if $X \to S$ is separated (see \cite[Proposition 3.48 (2)]{MR4665776}), if and only if for every DVR $R$ with fraction field $K$, any commutative diagram
\begin{equation}\label{1614}
\begin{tikzcd}
\mathrm{Spec}(K) \ar[r] \ar[d,hook] & X^\circ \ar[r,hook] & X \ar[d] \\
\mathrm{Spec}(R) \ar[urr,dashed] \ar[rr] & & S
\end{tikzcd}
\end{equation}
of solid arrows admits at most one dotted arrow filling in (see \cite[\href{https://stacks.math.columbia.edu/tag/0CME}{Tag 0CME}]{stacks-project}). Our argument below is taken from the proof of \cite[Proposition 3.48 (2)]{MR4665776}. Given a commutative diagram \eqref{1614} of solid arrows and suppose that there exist two dotted arrows filling in, i.e.,
\[
f_1, f_2: \mathrm{Spec}(R) \to X \text{ such that } u:=f_1|_K=f_2|_K: \mathrm{Spec}(K) \to X^\circ,
\]
then we claim that $f_1=f_2$. This equality can be checked up to finite extensions of DVRs. Since the adequate moduli space $\mathscr{X} \to X$ is universally closed (see \cite[Theorem 5.3.1 (2)]{MR3272912}), up to a finite extension of DVRs we may choose liftings
\begin{itemize}
\item
$\tilde{u}: \mathrm{Spec}(K) \to \mathscr{X}^\circ$ of $u$.
\item
$\tilde{f}_1,\tilde{f}_2: \mathrm{Spec}(R) \to \mathscr{X}$ of $f_1,f_2$ respectively such that $\tilde{f}_1|_K=\tilde{f}_2|_K=\tilde{u}$.
\end{itemize}
Therefore $\tilde{f}_1 \cup \tilde{f}_2$ defines a commutative diagram
\[
\begin{tikzcd}
\overline{\mathrm{ST}}_R-\{0\} \ar[r,"\tilde{f}_1 \cup \tilde{f}_2"] \ar[d,hook] & \mathscr{X} \ar[d] \\
\overline{\mathrm{ST}}_R \ar[ur,dashed,"\exists !"'] \ar[r] & S
\end{tikzcd}
\]
of solid arrows such that $(\tilde{f}_1 \cup \tilde{f}_2)|_K=\tilde{u}: \mathrm{Spec}(K) \hookrightarrow \overline{\mathrm{ST}}_R-\{0\} \to \mathscr{X}$ factors through $h: \mathscr{X}^\circ \to \mathscr{X}$, by assumption there exists a unique dotted arrow filling in. Since $\overline{\mathrm{ST}}_R \to \mathrm{Spec}(R)$ is a good moduli space, it is universal for maps to algebraic spaces (see \cite[Theorem 6.6]{MR3237451}), the extended morphism $\overline{\mathrm{ST}}_R \to \mathscr{X}$ descends to a unique morphism $\mathrm{Spec}(R) \to X$ which is equal to both $f_1$ and $f_2$.
\subsubsection{Case (2)}
Suppose $f: \mathscr{X} \to \mathscr{Y}$ is S-complete relative to $h$. The idea is to reformulate S-completeness as a lifting problem for DVRs (as in the proof of \cite[Proposition 3.42]{MR4665776}) and then reduces to the case of algebraic spaces.

The stack $\overline{\mathrm{ST}}_R$ can be viewed as a local model of the quotient stack $[\mathbf{A}^2/\mathbf{G}_m]$, where $\mathbf{A}^2=\mathrm{Spec}(k[x,y])$ such that $x,y$ have $\mathbf{G}_m$-weights $1,-1$ respectively. To be precise, $\overline{\mathrm{ST}}_R$ is the base change of the good moduli space $[\mathbf{A}^2/\mathbf{G}_m] \to \mathbf{A}^1=\mathrm{Spec}(k[xy])$ along the morphism $\mathrm{Spec}(R) \to \mathbf{A}^1$ corresponding to $xy \to \pi$, i.e., we have a Cartesian diagram
\begin{equation}\label{1508-5}
\begin{tikzcd}
\mathrm{Spec}(R[x,y]/xy-\pi) \ar[d] \ar[r] & \mathbf{A}^2=\mathrm{Spec}(k[x,y]) \ar[d,"{\mathbf{G}_m\text{-tor}}"] \\
\overline{\mathrm{ST}}_R \ar[r] \ar[d] & \left[\mathbf{A}^2/\mathbf{G}_m\right] \ar[d,"\text{gms}"] \ar[ul,phantom,"\lrcorner"] \\
\mathrm{Spec}(R) \ar[r,"xy \mapsto \pi"'] & \mathbf{A}^1=\mathrm{Spec}(k[xy]). \ar[ul,phantom,"\lrcorner"]
\end{tikzcd}
\end{equation}
Denote by $\mathbf{S}(\mathscr{X}):=\underline{\mathrm{Map}}_{\mathbf{A}^1}([\mathbf{A}^2/\mathbf{G}_m],\mathscr{X} \times \mathbf{A}^1) \to \mathbf{A}^1$ the mapping stack over $\mathbf{A}^1$. Since $[\mathbf{A}^2/\mathbf{G}_m] \to \mathbf{A}^1$ is a flat good moduli space, \cite[Theorem 5.10]{MR4088350} implies that the mapping stack $\mathbf{S}(\mathscr{X})$ is an algebraic stack, locally of finite type over $\mathbf{A}^1$. The pre-image of $\mathbf{A}^1-\{0\}$ in $\mathbf{S}(\mathscr{X})$ is $\mathscr{X} \times (\mathbf{A}^1-\{0\})$ and the fiber of $\mathbf{S}(\mathscr{X}) \to \mathbf{A}^1$ over $0$ is 
\[
\underline{\mathrm{Map}}_S(\Theta,\mathscr{X}) \times_{\underline{\mathrm{Map}}_S(\mathrm{B}\mathbf{G}_{m},\mathscr{X})} \underline{\mathrm{Map}}_S(\Theta,\mathscr{X}).
\]
Moreover we have the evaluation map
\[
\mathrm{ev}(f)_{(x,1),(1,y)}: \mathbf{S}(\mathscr{X}) \to \mathbf{S}(\mathscr{Y}) \times_{\mathscr{Y} \times \mathscr{Y}} \mathscr{X} \times \mathscr{X}
\]
given by on the first factor $\mathbf{S}(\mathscr{X}) \to \mathbf{S}(\mathscr{Y})$ composition and on the second factor $\mathbf{S}(\mathscr{X}) \to \mathscr{X} \times \mathscr{X}$ evaluation on the two sections $\mathrm{ev}_{(x,1)},\mathrm{ev}_{(1,y)}: \mathbf{A}^1 \to [\mathbf{A}^2/\mathbf{G}_m]$ of the good moduli space $[\mathbf{A}^2/\mathbf{G}_m] \to \mathbf{A}^1$. By \eqref{1508-5} a morphism $\overline{\mathrm{ST}}_R \to \mathscr{X}$ is a section in the diagram
\[
\begin{tikzcd}
& \mathbf{S}(\mathscr{X}) \ar[d] \\
\mathrm{Spec}(R) \ar[r,"xy \mapsto \pi"'] \ar[ur,dashed] & \mathbf{A}^1.
\end{tikzcd}
\]
Then the original lifting problem of S-completeness \eqref{D1} translates into the following lifting problem:
\begin{equation}\label{D2}
\begin{tikzcd}
\mathrm{Spec}(K) \ar[r] \ar[d,hook] & \mathbf{S}(\mathscr{X}) \ar[d,"{\mathrm{ev}(f)_{(x,1),(1,y)}}"] \\
\mathrm{Spec}(R) \ar[r] \ar[ur,dashed,"\exists!"'] & \mathbf{S}(\mathscr{Y}) \times_{\mathscr{Y}^2} \mathscr{X}^2
\end{tikzcd}
\end{equation}
and $\mathrm{Spec}(K) \hookrightarrow \overline{\mathrm{ST}}_R-\{0\} \to \mathscr{X}$ in \eqref{D1} factors through $h: \mathscr{X}^\circ \to \mathscr{X}$ if and only if $\mathrm{Spec}(K) \to \mathbf{S}(\mathscr{X})$ in \eqref{D2} factors through $\mathbf{S}(h): \mathbf{S}(\mathscr{X}^\circ) \to \mathbf{S}(\mathscr{X})$. The morphism $\mathrm{ev}(f)_{(x,1),(1,y)}$ is representable (as in \cite[Lemma 3.7 (1)]{MR4665776}) but not necessarily quasi-compact (see \cite[Remark 3.8]{MR4665776}).

Replacing $\mathscr{X}^\circ$ by an open dense substack of $\mathscr{X}$ in the image of the dominant morphism $h: \mathscr{X}^\circ \to \mathscr{X}$, we may assume that $h$ is an open dense immersion.
\begin{enumerate}
\item[(a)] 
First we prove the claim for a quotient stack $\mathscr{X}=[X/\mathbf{G}_m]$. Then $\mathscr{X}^\circ \cong [X^\circ/\mathbf{G}_m]$ for some $\mathbf{G}_m$-invariant open dense subset $X^\circ \subseteq X$. Note that any morphism $[\mathbf{A}^2/\mathbf{G}_m] \to \mathscr{X}$ gives a commutative diagram with Cartesian squares
\[
\begin{tikzcd}
& \mathbf{A}^2_d \ar[dr] \\
\mathbf{A}^2 \times \mathbf{G}_m \ar[d] \ar[r] & \left[\mathbf{A}^2 \times \mathbf{G}_m/\mathbf{G}_m\right] \ar[r] \ar[d] \ar[u,"\text{gms}"] & X \ar[d] \\
\mathbf{A}^2 \ar[r,"\mathbf{G}_m\text{-tor}"'] & \left[\mathbf{A}^2/\mathbf{G}_m\right] \ar[r] \ar[ul,phantom,"\lrcorner"] & \mathscr{X} \ar[ul,phantom,"\lrcorner"],
\end{tikzcd}
\]
where $\mathbf{A}^2 \times \mathbf{G}_m=\mathrm{Spec}(k[x,y,t,t^{-1}])$ such that $x,y,t$ have $\mathbf{G}_m$-weights $1,-1,d$ respectively for some positive integer $d$. Since $X$ is an algebraic space, the morphism $[\mathbf{A}^2 \times \mathbf{G}_m/\mathbf{G}_m] \to X$ factors through the good moduli space 
\begin{align*}
[\mathbf{A}^2 \times \mathbf{G}_m/\mathbf{G}_m] \to \mathbf{A}^2_d:&=\mathrm{Spec}(k[xy,x^dt^{-1},y^dt]) \\
&=\mathrm{Spec}(k[\pi,x',y']/x'y'-\pi^d),
\end{align*}
and the induced morphism $\mathbf{A}^2_d \to X$ is $\mathbf{G}_m$-equivariant with respect to the $\mathbf{G}_m$-action on $\mathbf{A}^2_d$ such that $x',y'$ have $\mathbf{G}_m$-weights $d,-d$ respectively. We have a section $[\mathbf{A}^2/\mathbf{G}_m] \to [\mathbf{A}^2 \times \mathbf{G}_m/\mathbf{G}_m]$ and the resulting composition $[\mathbf{A}^2/\mathbf{G}_m] \to [\mathbf{A}^2 \times \mathbf{G}_m/\mathbf{G}_m] \to \mathbf{A}^2_d$ induces a morphism
\[
\underline{\mathrm{Map}}_{\mathbf{A}^1}^{\mathbf{G}_m}(\mathbf{A}^2_d,X \times \mathbf{A}^1) \to \underline{\mathrm{Map}}_{\mathbf{A}^1}([\mathbf{A}^2/\mathbf{G}_m],\mathscr{X} \times \mathbf{A}^1)=\mathbf{S}(\mathscr{X}).
\]
Since $\mathbf{A}^2_d \to \mathbf{A}^1$ is a flat good quotient, \cite[Corollary 5.11]{MR4088350} implies that the $\mathbf{G}_m$-equivariant mapping stack $\underline{\mathrm{Map}}_{\mathbf{A}^1}^{\mathbf{G}_m}(\mathbf{A}^2_d,X \times \mathbf{A}^1)$ is an algebraic space, locally of finite type over $\mathbf{A}^1$. The above argument shows that any morphism $\mathrm{Spec}(K) \to \mathbf{S}(\mathscr{X})$ lifts to $\underline{\mathrm{Map}}_{\mathbf{A}^1}^{\mathbf{G}_m}(\mathbf{A}^2_d,X \times \mathbf{A}^1)$ for some uniquely determined positive integer $d$, and the same would hold for any extension $\mathrm{Spec}(R) \to \mathbf{S}(\mathscr{X})$. Then the diagram \eqref{D2} has a lifting if and only if for every integer $d>0$ the following lifting problem:
\begin{equation}\label{D3}
\begin{tikzcd}
\mathrm{Spec}(K) \ar[r] \ar[dd,hook] & \underline{\mathrm{Map}}_{\mathbf{A}^1}^{\mathbf{G}_m}(\mathbf{A}^2_d,X \times \mathbf{A}^1) \ar[d] \\
& \mathbf{S}(\mathscr{X}) \ar[d,"{\mathrm{ev}(f)_{(x,1),(1,y)}}"] \\
\mathrm{Spec}(R) \ar[r] \ar[uur,dashed,"\exists!"'] & \mathbf{S}(\mathscr{Y}) \times_{\mathscr{Y}^2} \mathscr{X}^2,
\end{tikzcd}
\end{equation}
has a lifting. However, the morphism induced by the open dense immersion $X^\circ \to X$
\[
\underline{\mathrm{Map}}_{\mathbf{A}^1}^{\mathbf{G}_m}(\mathbf{A}^2_d,X^\circ \times \mathbf{A}^1) \to \underline{\mathrm{Map}}_{\mathbf{A}^1}^{\mathbf{G}_m}(\mathbf{A}^2_d,X \times \mathbf{A}^1)
\]
is an open immersion (as in \cite[Proposition 2.1.13 (ii)]{Drinfeld-Gm}) but not necessarily dense as the target might be reducible (see, e.g. the example in \cite[Footnote 12 in page 22]{Drinfeld-Gm}). To remedy this note that the good quotient $\mathbf{A}^2_d \to \mathbf{A}^1$ has two sections given by non-vanishings of $x',y'$ respectively and they equip $\underline{\mathrm{Map}}_{\mathbf{A}^1}^{\mathbf{G}_m}(\mathbf{A}^2_d,X \times \mathbf{A}^1)$ with a morphism
\[
p: \underline{\mathrm{Map}}_{\mathbf{A}^1}^{\mathbf{G}_m}(\mathbf{A}^2_d,X \times \mathbf{A}^1) \to \mathbf{A}^1 \times X \times X.
\]
Similarly as in \cite[Proposition 2.1.8]{Drinfeld-Gm}, the pre-image of $\mathbf{A}^1-\{0\}$ under the structure morphism $\underline{\mathrm{Map}}_{\mathbf{A}^1}^{\mathbf{G}_m}(\mathbf{A}^2_d,X \times \mathbf{A}^1) \to \mathbf{A}^1$ is isomorphic to, via $p$, the graph of the morphism
\[
\mathrm{act}_d: \mathbf{G}_m \times X \xrightarrow{(t \mapsto t^d,\mathrm{id}_X)} \mathbf{G}_m \times X \xrightarrow{\text{act}} X,
\]
i.e.,
\[
\begin{tikzcd}
\Gamma(\mathrm{act}_d) \ar[r,hook,"\circ" marking] \ar[d,hook,"/" marking] & \underline{\mathrm{Map}}_{\mathbf{A}^1}^{\mathbf{G}_m}(\mathbf{A}^2_d,X \times \mathbf{A}^1) \ar[d,"p"] \\
\mathbf{G}_m \times X \times X \ar[r,hook,"\circ" marking] \ar[d] & \mathbf{A}^1 \times X \times X \ar[d] \ar[ul,phantom,"\lrcorner"] \\
\mathbf{G}_m \ar[r,hook,"\circ" marking] & \mathbf{A}^1. \ar[ul,phantom,"\lrcorner"]
\end{tikzcd}
\]
Let 
\[
\underline{\mathrm{Map}}_{\mathbf{A}^1}^{\mathbf{G}_m}(\mathbf{A}^2_d,X \times \mathbf{A}^1)^{\flat} \subseteq \underline{\mathrm{Map}}_{\mathbf{A}^1}^{\mathbf{G}_m}(\mathbf{A}^2_d,X \times \mathbf{A}^1)
\]
be the closed subspace given by the closure of $\Gamma(\mathrm{act}_d) \subseteq \mathbf{G}_m \times X \times X$ in $\mathbf{A}^1 \times X \times X$. Note that any morphism $\mathrm{Spec}(K) \to \underline{\mathrm{Map}}_{\mathbf{A}^1}^{\mathbf{G}_m}(\mathbf{A}^2_d,X \times \mathbf{A}^1)$ factors through $\underline{\mathrm{Map}}_{\mathbf{A}^1}^{\mathbf{G}_m}(\mathbf{A}^2_d,X \times \mathbf{A}^1)^{\flat}$, and the same would hold for any extension $\mathrm{Spec}(R) \to \underline{\mathrm{Map}}_{\mathbf{A}^1}^{\mathbf{G}_m}(\mathbf{A}^2_d,X \times \mathbf{A}^1)$. Then the lifting problem \eqref{D3} translates into the following lifting problem:
\begin{equation}\label{D4}
\begin{tikzcd}
\mathrm{Spec}(K) \ar[r] \ar[ddd,hook] & \underline{\mathrm{Map}}_{\mathbf{A}^1}^{\mathbf{G}_m}(\mathbf{A}^2_d,X \times \mathbf{A}^1)^\flat \ar[d] \\
& \underline{\mathrm{Map}}_{\mathbf{A}^1}^{\mathbf{G}_m}(\mathbf{A}^2_d,X \times \mathbf{A}^1) \ar[d] \\
& \mathbf{S}(\mathscr{X}) \ar[d,"{\mathrm{ev}(f)_{(x,1),(1,y)}}"] \\
\mathrm{Spec}(R) \ar[r] \ar[uuur,dashed,"\exists!"'] & \mathbf{S}(\mathscr{Y}) \times_{\mathscr{Y}^2} \mathscr{X}^2.
\end{tikzcd}
\end{equation}
The morphism induced by the open dense immersion $X^\circ  \to X$
\[
\underline{\mathrm{Map}}_{\mathbf{A}^1}^{\mathbf{G}_m}(\mathbf{A}^2_d,X^\circ \times \mathbf{A}^1)^{\flat} \to \underline{\mathrm{Map}}_{\mathbf{A}^1}^{\mathbf{G}_m}(\mathbf{A}^2_d,X \times \mathbf{A}^1)^{\flat}
\]
is an open dense immersion as the target is irreducible. As the right vertical arrow in \eqref{D4} is representable and quasi-compact, this reduces to the lifting problem for a quasi-compact morphism of algebraic spaces, which can be checked relative to the open dense immersion $\underline{\mathrm{Map}}_{\mathbf{A}^1}^{\mathbf{G}_m}(\mathbf{A}^2_d,X^\circ \times \mathbf{A}^1)^{\flat} \to \underline{\mathrm{Map}}_{\mathbf{A}^1}^{\mathbf{G}_m}(\mathbf{A}^2_d,X \times \mathbf{A}^1)^{\flat}$ by \cite[\href{https://stacks.math.columbia.edu/tag/0CME}{Tag 0CME}]{stacks-project}. Thus, in the lifting problem \eqref{D2} we can assume that the morphism $\mathrm{Spec}(K) \to \mathbf{S}(\mathscr{X})$ factors through $\mathbf{S}(h): \mathbf{S}(\mathscr{X}^\circ) \to \mathbf{S}(\mathscr{X})$, as desired.
\item[(b)]
In general, choose a separated representable \'{e}tale cover of $\mathscr{X}$ by quotient stacks $[X/\mathbf{G}_m]$. Since $h: \mathscr{X}^\circ \to \mathscr{X}$ is an open dense immersion, for each chart $[X/\mathbf{G}_m] \to \mathscr{X}$ we have a Cartesian diagram
\[
\begin{tikzcd}
\left[X^\circ/\mathbf{G}_m\right] \ar[r,"\hat{h}"] \ar[d] & \left[X/\mathbf{G}_m\right] \ar[d] \\
\mathscr{X}^\circ \ar[r,"h"'] & \mathscr{X} \ar[ul,phantom,"\lrcorner"]
\end{tikzcd}
\]
for some $\mathbf{G}_m$-invariant open dense subset $X^\circ \subseteq X$. Since $[X/\mathbf{G}_m] \to \mathscr{X}$ is separated (and hence S-complete by \cite[Remark 3.40]{MR4665776}) and $f: \mathscr{X} \to \mathscr{Y}$ is S-complete relative to $h$, the composition $[X/\mathbf{G}_m] \to \mathscr{X} \to \mathscr{Y}$ is S-complete relative to $\hat{h}$. Then by the claim in (a) the composition $[X/\mathbf{G}_m] \to \mathscr{X} \to \mathscr{Y}$ is S-complete. To conclude, it suffices to show in any lifting problem of S-completeness \eqref{D1}, the morphism $\overline{\mathrm{ST}}_R-\{0\} \to \mathscr{X}$ lifts to some chart $[X/\mathbf{G}_m]$, which can be checked up to extensions of DVRs by \cite[Proposition 3.41 (2)]{MR4665776}. For this it suffices to find a point $\eta \in |\mathscr{X}|$ that is a common specialization of the images of the closed points under two maps $\mathrm{Spec}(R) \to \mathscr{X}$, so that the image of $\overline{\mathrm{ST}}_R-\{0\} \to \mathscr{X}$ is contained in the image of some $[X/\mathbf{G}_m] \to \mathscr{X}$. Indeed, we can modify two maps $\mathrm{Spec}(R) \to \mathscr{X}$ to $\mathrm{Spec}(R') \to \mathscr{X}$ (for some extension of DVRs $R \subseteq R'$) such that the images of the closed points are the same but $\mathrm{Spec}(K')$ maps to $\mathscr{X}^\circ \subseteq \mathscr{X}$. By hypothesis there exists a morphism $\overline{\mathrm{ST}}_{R'} \to \mathscr{X}$ extending $\mathrm{Spec}(R') \cup_{\mathrm{Spec}(K')} \mathrm{Spec}(R')=\overline{\mathrm{ST}}_{R'}-\{0\} \to \mathscr{X}$, which implies that the images of the closed points under two maps $\mathrm{Spec}(R') \to \mathscr{X}$ (and hence two maps $\mathrm{Spec}(R) \to \mathscr{X}$) have a common specialization $\eta \in |\mathscr{X}|$. \hfill $\square$
\end{enumerate}
\subsection{Recollections on $\mathbf{G}_m$-actions}
Let $X$ be a variety a field $k$ with a $\mathbf{G}_m$-action $\sigma: \mathbf{G}_m \times X \to X$. Let $\mathbf{K}/k$ be a field. For any point $x \in X(\mathbf{K})$, if its orbit map 
\[
\sigma_x: \mathbf{G}_{m,\mathbf{K}}=\mathbf{G}_m \times \mathrm{Spec}(\mathbf{K}) \xrightarrow{(\mathrm{id},x)} \mathbf{G}_m \times X \xrightarrow{\sigma} X
\]
extends to $\overline{\sigma}_x: \mathbf{P}^1_{\mathbf{K}} \to X$ (e.g. if $X$ is proper), then we call $\overline{\sigma}_x$ the complete orbit map of $x$ and write $x^-:=\overline{\sigma}_x(0)$ and $x^+:=\overline{\sigma}_x(\infty)$. Both of them are $\mathbf{G}_m$-fixed points. 
Consider the following functors of $\mathbf{G}_m$-equivariant morphisms:
\[
X^0:=\underline{\mathrm{Map}}_k^{\mathbf{G}_m}(\mathrm{Spec}(k),X) \text{ and } X^{\pm}:=\underline{\mathrm{Map}}_k^{\mathbf{G}_m}(\mathbf{A}^{\pm 1},X),
\]
where $\mathbf{A}^{\pm 1}=\mathrm{Spec}(k[a])$ is the usual affine line such that $a$ has $\mathbf{G}_m$-weight $\pm 1$ respectively. By \cite[Proposition 1.2.2 and Corollary 1.4.3]{Drinfeld-Gm} these functors are represented by separated schemes of finite type over $k$. Furthermore, there are natural morphisms: a closed immersion (see \cite[Proposition 1.2.2]{Drinfeld-Gm}) $X^0 \to X$, a monomorphism $\mathrm{ev}_1: X^\pm \to X$ (given by evaluation at 1) and an affine (see \cite[Theorem 1.4.2]{Drinfeld-Gm}) morphism $\mathrm{ev}_0: X^\pm \to X^0$ (given by evaluation at 0). Set-theoretically $X^0 \subseteq X$ is the subset of $\mathbf{G}_m$-fixed points and (under the morphism $\mathrm{ev}_1$) $X^\pm=\{x \in X: x^\pm \in X\} \subseteq X$ such that the morphism $\mathrm{ev}_0$ is given by $x \mapsto x^\pm$ (see \cite[\S 1.3.3]{Drinfeld-Gm}). Let $X^0=\bigsqcup_{i \in \pi_0(X^0)} X_i$ be the connected components, then
\[
X_i^\pm:=\mathrm{ev}_0^{-1}(X_i)=\{x \in X: x^\pm \in X_i\} \subseteq X^\pm
\]
is constructible (see, e.g. \cite[Lemma 5 and Remark 6]{MR0704987}) and $X^\pm=\bigsqcup_{i \in \pi_0(X^0)} X_i^\pm$.
\begin{rmk}\label{0948-4}
If $X$ is normal, then it has a $\mathbf{G}_m$-invariant affine open cover by Sumihiro's theorem (see \cite[Corollary 2]{MR337963} and \cite[Theorem 3.8 and Corollary 3.11]{MR387294}). In this case the fixed points locus $X^0$ is also normal (see \cite[Proposition 7.4]{MR332805}) so its connected components $X_i$ are irreducible.
\end{rmk}
One feature about $\mathbf{G}_m$-actions is that they form a canonical flow on $X$. The flow-lines define a natural relation on the set $\pi_0(X^{0})$ of fixed point components.
\begin{defn}[\cite{MR704983}, Definition 1.1]
For any $i,j \in \pi_0(X^0)$, we say
\begin{enumerate}
\item
$X_i$ is \emph{directly less than} $X_j$, written as $X_i <_d X_j$, if there exists a point $x \in X$ such that $x^- \in X_i$ and $x^+ \in X_j$, or equivalently, if $X_i^- \cap X_j^+ \neq \emptyset$.
\item
$X_i$ is \emph{less than} $X_j$, written as $X_i<X_j$, if $X_i<_d \cdots <_d X_j$.
\end{enumerate}
\end{defn}
For $\mathbf{G}_m$-actions there are two distinguished fixed point components.
\begin{lem}[\cite{MR709583}, Lemma A.1 if $k=\mathbf{C}$]\label{source-sink}
Let $X$ be a proper variety over a field $k$ with a $\mathbf{G}_m$-action. There exists a unique fixed point component $X_{\min}$ (resp., $X_{\max}$), called the source (resp., the sink) of $X$, characterized by the property that $X_{\min}^- \subseteq X$ (resp., $X_{\max}^+ \subseteq X$) is dense. If $X$ is in addition normal, then $X_{\min}^- \subseteq X$ (resp., $X_{\max}^+ \subseteq X$) is open dense.
\end{lem}
\begin{proof}
Since $X$ is proper, we have $X=X^\pm=\bigsqcup_{i \in \pi_0(X^0)} X_i^\pm$. The source and the sink of $X$ are singled out by looking at where the generic point of $X$ locates according to the minus or plus decomposition, i.e., $X_{\min}^- \subseteq X$ and $X_{\max}^+ \subseteq X$ are dense. If $X$ is normal, it has a $\mathbf{G}_m$-invariant affine open cover by Sumihiro's theorem, so $X_{\min}^- \subseteq X$ and $X_{\max}^+ \subseteq X$ are open dense by \cite[Theorem 9]{MR0704987}.
\end{proof}
\section{Characterization for properness}\label{Geom-char}
In this section, we translate the existence criteria of proper moduli spaces in \cite{MR4665776} for quotient stacks from $\mathbf{G}_m$-actions.
\begin{thm}[\cite{MR4665776}, Theorem 5.4]\label{thm0}
Let $\mathscr{X}$ be a locally reductive algebraic stack, of finite presentation and with affine diagonal over a field $k$. Then $\mathscr{X}$ admits a proper adequate moduli space if and only if 
\begin{itemize}
\item[$(\Theta)$]
$\mathscr{X}$ is $\Theta$-reductive,
\item[$(\mathrm{S})$]
$\mathscr{X}$ is S-complete, and
\item[$(\mathrm{E})$]
$\mathscr{X}$ satisfies the existence part of the valuative criterion for properness.
\end{itemize}
\end{thm}
\begin{rmk}
The quotient stack $[X/\mathbf{G}_m]$ is locally reductive for any normal variety $X$ by Sumihiro's theorem. In this case, an adequate moduli space of $[X/\mathbf{G}_m]$, if exists, is a good moduli space.
\end{rmk}
Let us briefly explain how to get a description of open subsets with proper good quotients using this existence criteria. Let $X$ be a normal proper variety with a $\mathbf{G}_m$-action. If $U \subseteq X$ is a $\mathbf{G}_m$-invariant open subset with a proper good quotient, then the quotient stack $[U/\mathbf{G}_m]$ satisfies the conditions (S) and (E) by Theorem \ref{thm0}. By definition in the level of atlas this amounts to saying that $U$ intersects degenerations of all orbits in $U$ (condition (E)), with a unique closed point (condition (S)). By Theorem \ref{thm:S-com-new} and \cite[\href{https://stacks.math.columbia.edu/tag/0CQM}{Tag 0CQM}]{stacks-project}, for both conditions (E) and (S) it suffices to consider degenerations of the generic orbit closure, i.e., orbit closure of the generic point. A degeneration of the generic orbit closure flows from the source to the sink by Lemma \ref{source-sink}, and may have several components. 

\pagebreak

\begin{figure}[!ht]
\centering
\begin{tikzpicture}

\draw[->,thick] (0,0) -- (1,0.5);
\draw[->,thick] (1,0.5) -- (2.1,0.5);
\draw[->,thick] (2.1,0.5) -- (4.3,1.1);
\draw[->,thick] (4.3,1.1) -- (5.1,0.5);
\draw[->,thick] (5.1,0.5) -- (6,0.7);

\filldraw (0,0) circle (.05)
(1,0.5) circle (.05)
(2.1,0.5) circle (.05)
(4.3,1.1) circle (.05)
(5.1,0.5) circle (.05)
(6,0.7) circle (.05);

\node[left] at (0,0) {$X_{\min}$};
\node[right] at (6,0.7) {$X_{\max}$};

\end{tikzpicture}
\caption{A degeneration of the generic orbit closure}
\end{figure}
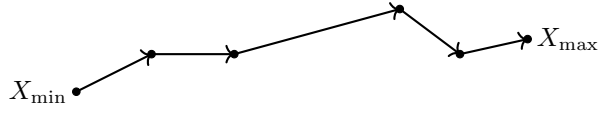

Since its intersection with $U$ is non-empty, $\mathbf{G}_m$-invariant, open and has a unique closed point, all possible configurations are as follows.
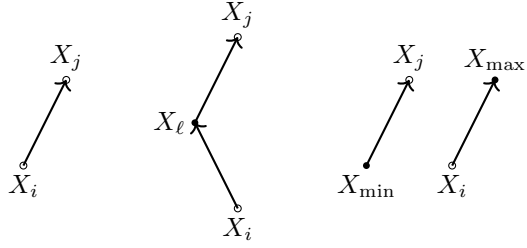
\begin{figure}[!ht]
\centering
\begin{tikzpicture}[scale=0.57]

\draw[->,thick] (0,0) -- (1,2);

\draw[->,thick] (5,-1) -- (4,1);
\draw[->,thick] (4,1) -- (5,3);

\draw[->,thick] (8,0) -- (9,2);

\draw[->,thick] (10,0) -- (11,2);

\draw (0,0) circle (.08)
(1,2) circle (.08)
(5,-1) circle (.08)
(5,3) circle (.08)
(9,2) circle (.08)
(10,0) circle (.08)
;

\fill (4,1) circle (.08)
(8,0) circle (.08)
(11,2) circle (.08)
;

\node[below] at (0,0) {$X_i$};
\node[above] at (1,2) {$X_j$};
\node[below] at (5,-1) {$X_i$};
\node[left] at (4,1) {$X_\ell$};
\node[above] at (5,3) {$X_j$};
\node[below] at (8,0) {$X_{\min}$};
\node[above] at (9,2) {$X_j$};
\node[below] at (10,0) {$X_i$};
\node[above] at (11,2) {$X_{\max}$};

\end{tikzpicture}
\caption{All possible configurations of the intersection}
\label{poss-fig}
\end{figure}

Notably, this phenomenon characterizes properness of the moduli space.
\begin{keyob}\label{2v}
The quotient stack $[U/\mathbf{G}_m]$ satisfies the conditions (E) and (S) if and only if $U$ intersects each degeneration of the generic orbit closure with one of the configurations listed above.
\end{keyob}
This expectation will be made precisely in Theorem \ref{summary}. Therefore degenerations of the generic orbit closure deserve a name and this is the prototype of the so-called smoothable maximal chain of orbits.
\subsection{Chain of orbits}
Let $X$ be a proper variety over a field $k$ with a $\mathbf{G}_m$-action. The notion ``chain of orbits'' is introduced as ``equivariant chain of projective lines (of negative weight)'' in \cite[\S2.B]{MR3758902}.
\begin{defn}\label{def:sm}
A \emph{chain of orbits} in $X$ is a $\mathbf{G}_m$-equivariant morphism
\[
f_{\mathbf{K}}: C_{\mathbf{K}}=\mathbf{P}^1_{\mathbf{K}} \cup_{\infty \sim 0} \cdots \cup_{\infty \sim 0} \mathbf{P}^1_{\mathbf{K}} \to X
\]
for some field $\mathbf{K}/k$, where the $\mathbf{G}_m$-action on $C_{\mathbf{K}}$ is such that on each $\mathbf{P}_{\mathbf{K}}^1=\mathrm{Proj}(\mathbf{K}[\alpha,\beta])$ (with $0=[1:0]$ and $\infty=[0:1]$) it induces the $\mathbf{G}_m$-action of some negative weight $w_i<0$, i.e., it is given by $t.\alpha=t^{w_i+d}\alpha$ and $t.\beta=t^d\beta$.
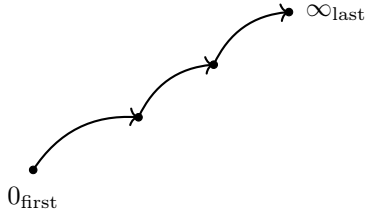
\begin{figure}[!ht]
\centering
\begin{tikzpicture}

\draw[->,thick,bend left=30] (0,0) to (1.4,0.7);
\draw[->,thick,bend left=30] (1.4,0.7) to (2.4,1.4);
\draw[->,thick,bend left=30] (2.4,1.4) to (3.4,2.1);

\filldraw (0,0) circle (.05)
(1.4,0.7) circle (.05)
(2.4,1.4) circle (.05)
(3.4,2.1) circle (.05);

\node[below=0.1cm] at (0,0) {$0_{\text{first}}$};
\node[right=0.1cm] at (3.4,2.1) {$\infty_{\text{last}}$};

\end{tikzpicture}
\caption{A chain of orbits}
\end{figure}
\end{defn}
\begin{defn}\label{key-defn:sm}
A chain of orbits $f_{\mathbf{K}}: C_{\mathbf{K}} \to X$ is said to be
\begin{enumerate}
\item 
\emph{maximal} if $f_\mathbf{K}(0_{\text{first}}) \in X_{\min}$ and $f_\mathbf{K}(\infty_{\text{last}}) \in X_{\max}$.
\item
\emph{U-smoothable} for a $\mathbf{G}_m$-invariant subset $U \subseteq X$, if there exist a DVR $R$ with fraction field $K$ and residue field $\mathbf{K}$, and a $\mathbf{G}_m$-equivariant commutative diagram
\begin{equation}\label{1310-2}
\begin{tikzcd}
\mathrm{Spec}(\mathbf{K}) \ar[d,hook] & C_\mathbf{K} \ar[r,"{f_\mathbf{K}}"] \ar[d,hook] \ar[l] & X \\
\mathrm{Spec}(R) \ar[ur,phantom,"\urcorner"] & C_R \ar[ur,"f_R"'] \ar[l]
\end{tikzcd}
\end{equation}
such that $C_K \cong \mathbf{P}^1_K$ and the generic fiber $f_K: C_K \to X$ is a complete orbit map of some point $x_K \in U(K)$. 

The chain of orbits $f_{\mathbf{K}}$ is \emph{smoothable} if it is $X$-smoothable. 
\end{enumerate}
\end{defn}
\begin{rmk}
There always exists a deformation $C_R$ of $C_\mathbf{K}$ such that the generic fiber is $\mathbf{P}^1_K$, e.g. blowing up $\mathbf{P}^1_R$ at $(\mathrm{Spec}(\mathbf{K}),\infty)$ and then iteratively blowing up $\infty$ in the exceptional $\mathbf{P}^1$'s until we get $C_\mathbf{K}$ in the special fiber. So for the diagram \eqref{1310-2} the main issue is the existence of a $\mathbf{G}_m$-equivariant lifting $f_R$.
\end{rmk}
\begin{rmk}
It is expected that every maximal chain of orbits in $X$ is smoothable, at least when $X$ is smooth projective. We know this is true for projective spaces using an explicit construction.
\end{rmk}
Smoothable maximal chains of orbits are effective as potential test objects since there are enough of them, i.e., $X$ is covered by such objects.
\begin{prop}[\cite{MR709583}, Theorem 0.1.2 and Corollary 0.2.4 if $k=\mathbf{C}$]\label{cover}
Let $X$ be a proper variety over a field $k$ with a $\mathbf{G}_m$-action. Let $U \subseteq X$ be a $\mathbf{G}_m$-invariant dense subset. For any point $x \in X$, there exists a $U$-smoothable maximal chain of orbits in $X$ passing through $x$.
\end{prop}
\begin{proof}
Let $\eta \in X$ be the generic point, then $\eta \in U \cap X_{\min}^- \cap X_{\max}^+$ by Lemma \ref{source-sink}. For any point $x \in X$, there exist a DVR $R$ with fraction field $K=k(\eta)$ and residue field $\mathbf{K}/k(x)$, and a morphism $x_R: \mathrm{Spec}(R) \to X$ mapping the generic point to $\eta$ and the closed point to $x$ (see \cite[\href{https://stacks.math.columbia.edu/tag/054F}{Tag 054F}]{stacks-project}). The complete orbit map $\overline{\sigma}_\eta: \mathbf{P}^1_K \to X$ of $\eta$ and the orbit map $\sigma_{x_R}: \mathbf{G}_{m,R} \to X$ of $x_R$ glue to a rational map $f: \mathbf{P}^1_R \dashrightarrow X$. Because $X$ is proper, after blowing-up some ideal $\mathscr{I} \subseteq \mathscr{O}_{\mathbf{P}^1_R}$ supported at $(\mathrm{Spec}(\mathbf{K}),0)$ and $(\mathrm{Spec}(\mathbf{K}),\infty)$, the map $f$ extends to a morphism $\tilde{f}: \mathrm{Bl}_{\mathscr{I}}(\mathbf{P}^1_R) \to X$, i.e.,
\[
\begin{tikzcd}
\Phi_\mathbf{K} \ar[r,hook] \ar[d] & \mathrm{Bl}_{\mathscr{I}}(\mathbf{P}^1_R) \ar[d] \ar[dr,"\tilde{f}"] \\
\mathbf{P}^1_\mathbf{K} \ar[r,hook] \ar[d] & \mathbf{P}^1_R \ar[r,dashed,"f"'] \ar[d] \ar[ul,phantom,"\lrcorner"] & X \\
\mathrm{Spec}(\mathbf{K}) \ar[r,hook] & \mathrm{Spec}(R). \ar[ul,phantom,"\lrcorner"]
\end{tikzcd}
\]
Since we blow up at $\mathbf{G}_m$-fixed points, the $\mathbf{G}_m$-action on $\mathbf{P}^1_R$ extends to the blow-up $\mathrm{Bl}_{\mathscr{I}}(\mathbf{P}^1_R)$ such that the morphism $\tilde{f}: \mathrm{Bl}_{\mathscr{I}}(\mathbf{P}^1_R) \to X$ is $\mathbf{G}_m$-equivariant. We claim that the composition $\Phi_\mathbf{K} \to \mathrm{Bl}_{\mathscr{I}}(\mathbf{P}^1_R) \to X$ can be refined to be a chain of orbits. Then we are done since then it is $U$-smoothable maximal and pass through $x$. Our proof below adapts the computations in \cite[Lemma 2.1]{MR3758902}.

Indeed, it suffices to treat the affine case that the ideal $I \subseteq \mathscr{O}_{\mathbf{A}^1_R}=R[y]$ is supported at $(y,\pi)$, where $\pi \in R$ is a uniformizer. There exists an integer $n>0$ such that $(y,\pi)^n \subseteq I$ and since $I$ is $\mathbf{G}_m$-invariant it is homogeneous with respect to the grading for which $y$ has $\mathbf{G}_m$-weight $1$ and $\pi$ has $\mathbf{G}_m$-weight $0$. Every homogeneous generator of $I$ is of the form $y^d\pi^m$ for some integer $m \geq 0$ so that $I$ is monomial. 

Since $(y,\pi)^n \subseteq I$ we may write $I=(y^a,\pi^b,y^{a_i}\pi^{b_i})_{i=1,\ldots,N}$ with $a<n, b<n, a_i+b_i<n$. This ideal becomes principal after successively blowing-up $\infty$ and then blowing-up $\infty$ in the exceptional $\mathbf{P}^1$'s: Blowing-up $(y,\pi)$ we get charts with coordinates $(y,\pi) \mapsto (y'\pi,\pi)$ and $(y,\pi) \mapsto (y,y\pi')$. Then the $\mathbf{G}_m$-weights of $(y',\pi)$ are $(1,0)$ and the $\mathbf{G}_m$-weights of $(y,\pi')$ are $(0,-1)$.

In the first chart the proper transform of $I$ is $(y'^{a}\pi^a,\pi^{b},y'^{a_i}\pi^{a_i+b_i})_{i=1,\ldots,N}$. This ideal is principal if $b=1$ and otherwise equals to an ideal of the form
\[
\pi^c(\pi^{b-c},\text{ mixed monomials of lower total degree}).
\]
A similar computation works in the other chart. By induction this shows that the ideal will become principal after finitely many blow-ups and that in each chart the coordinates $(y^{(i)},\pi^{(i)})$ have $\mathbf{G}_m$-weights $(w_i,v_i)$ with $w_i>v_i$.
\end{proof}
As a result, all fixed point components locate between the source and the sink.
\begin{cor}[\cite{MR709583}, Corollary 0.2.5 if $k=\mathbf{C}$]\label{well-order}
Let $X$ be a proper variety over a field $k$ with a $\mathbf{G}_m$-action. Then $X_{\min}<X_i<X_{\max}$ for each $i \in \pi_0(X^0)$. 
\end{cor}
We can now make \textbf{Expectation} precise. Let $U \subseteq X$ be a $\mathbf{G}_m$-invariant open dense subset. The following subset 
\[
V:=U \cap X_{\min}^- \cap X_{\max}^+ \subseteq U,
\]
consisting of points in $U$ flowing from the source to the sink, is $\mathbf{G}_m$-invariant and dense in $U$ by Lemma \ref{source-sink}. Hereafter, replacing $V$ by a $\mathbf{G}_m$-invariant open subset if necessary, we may assume that $V \subseteq U$ is open dense (which is automatic if $X$ is normal, see Lemma \ref{source-sink}). Then the conditions (S) and (E) for the quotient stack $\mathscr{U}=[U/\mathbf{G}_m]$ can be checked relative to the open dense substack $\mathscr{V}:=[V/\mathbf{G}_m] \subseteq \mathscr{U}$ by Theorem \ref{thm:S-com-new} and \cite[\href{https://stacks.math.columbia.edu/tag/0CQM}{Tag 0CQM}]{stacks-project}. This will give a geometric criterion when the quotient stack $\mathscr{U}$ satisfies (E) or (S). 
\subsection{Consequence of (E)}
The (E)-part of \textbf{Expectation} holds.
\begin{prop}[\cite{MR709583}, Lemma 1.2 if $k=\mathbf{C}$]\label{f1}
Let $X$ be a proper variety over a field $k$ with a $\mathbf{G}_m$-action. For any $\mathbf{G}_m$-invariant open subset $U \subseteq X$, the following are equivalent:
\begin{enumerate}
\item
The quotient stack $\mathscr{U}=[U/\mathbf{G}_m]$ satisfies the condition $\mathrm{(E)}$.
\item
The image of any $U$-smoothable maximal chain of orbits in $X$ intersects $U$.
\end{enumerate}
\end{prop}
\begin{proof}
By \cite[\href{https://stacks.math.columbia.edu/tag/0CQM}{Tag 0CQM}]{stacks-project}, the stack $\mathscr{U}$ satisfies the condition (E) if and only if 
\[
(1') \quad \mathscr{U} \text{ satisfies the condition (E) relative to } \mathscr{V} \subseteq \mathscr{U}.
\]
\begin{enumerate}
\item[$(1') \Rightarrow (2)$.]
For any $U$-smoothable maximal chain of orbits $f_{\mathbf{K}}: C_{\mathbf{K}} \to X$ in $X$, we show $U \cap f_{\mathbf{K}}(C_{\mathbf{K}}) \neq \emptyset$. By definition there exist a DVR $R$ with fraction field $K$ and residue field $\mathbf{K}$, and a $\mathbf{G}_m$-equivariant commutative diagram \eqref{1310-2} such that the generic fiber $f_K: C_K \to X$ is a complete orbit map of some point $x_K \in V(K)$. By hypothesis for the morphism 
\[
\mathrm{Spec}(K) \xrightarrow{x_K} V \twoheadrightarrow \mathscr{V} \subseteq \mathscr{U}
\]
there exists a morphism $x'_R: \mathrm{Spec}(R) \to U$ such that $x'_K=\eta_K.x_K$ for some $\eta_K \in \mathbf{G}_m(K)$. Consider the following commutative diagram
\[
\begin{tikzcd}
C_R \ar[r,"f_R"] & X & \mathrm{Bl}_{\mathscr{I}}(\mathbf{P}^1_R) \ar[l,"u_R"'] \ar[d]  \\
C_K \ar[u,hook] \ar[r,"f_K"] & U \ar[u,hook] & \mathbf{P}^1_R \ar[ul,dashed] \\
\mathbf{P}^1_K \ar[u,equal] & \mathbf{G}_{m,K} \ar[l,hook] \ar[r,hook] & \mathbf{G}_{m,R} \ar[ul,"{\sigma_{x'_R}}"',near end] \ar[u,hook]
\end{tikzcd}
\]
where $\mathscr{I} \subseteq \mathscr{O}_{\mathbf{P}^1_R}$ is some ideal supported at $(\mathrm{Spec}(\mathbf{K}),0)$ and $(\mathrm{Spec}(\mathbf{K}),\infty)$ such that the rational map $\mathbf{P}_R^1 \dashrightarrow X$ extends to a morphism $\mathrm{Bl}_{\mathscr{I}}(\mathbf{P}^1_R) \to X$. Note that $\mathrm{Im}(u_R)=\mathrm{Im}(f_R)$ since both contain the image of $\mathbf{G}_{m,K}$ as a dense subset. Denote by $\Phi_{\mathbf{K}} \subseteq \mathrm{Bl}_{\mathscr{I}}(\mathbf{P}^1_R)$ the special fiber, then $u_{\mathbf{K}}(\Phi_{\mathbf{K}})=f_{\mathbf{K}}(C_{\mathbf{K}})$ and hence 
\[
U \cap f_{\mathbf{K}}(C_{\mathbf{K}})=U \cap u_{\mathbf{K}}(\Phi_{\mathbf{K}}) \neq \emptyset.
\]
\item[$(2) \Rightarrow (1')$.]
For every DVR $R$ with fraction field $K$, any commutative diagram
\[
\begin{tikzcd}
\mathrm{Spec}(K) \ar[d,hook] \ar[r] & \mathscr{V} \ar[r,hook] & \mathscr{U} \ar[d] \\
\mathrm{Spec}(R) \ar[urr,dashed] \ar[rr] & & \mathrm{Spec}(k)
\end{tikzcd}
\]
of solid arrows, we show there exists a dotted arrow filling in. Indeed, the morphism $\mathrm{Spec}(K) \to \mathscr{V}$ lifts to the altas $x_K: \mathrm{Spec}(K) \to V$ and it extends to a morphism $x_R: \mathrm{Spec}(R) \to X$ since $X$ is proper. Consider the following commutative diagram
\[
\begin{tikzcd}
\mathrm{Bl}_{\mathscr{I}}(\mathbf{P}^1_R) \ar[d] \ar[dr,"u_R"] \\
\mathbf{P}^1_R \ar[r,dashed] & X \\
\mathbf{G}_{m,R} \ar[u,hook] \ar[ur,"{\sigma_{x_R}}"']
\end{tikzcd}
\]
As in the proof of Proposition \ref{cover}, the special fiber $u_\mathbf{K}: \Phi_\mathbf{K} \to X$ of $u_R$ can be refined as a chain of orbits in $X$, which is $U$-smoothable and maximal by construction. Then by hypothesis $U \cap u_\mathbf{K}(\Phi_\mathbf{K}) \neq \emptyset$. Choose a section $s: \mathrm{Spec}(R) \to \mathrm{Bl}_{\mathscr{I}}(\mathbf{P}^1_R)$ such that the composition 
\[
\zeta: \mathrm{Spec}(R) \xrightarrow{s} \mathrm{Bl}_{\mathscr{I}}(\mathbf{P}^1_R) \xrightarrow{u_R} X
\]
maps the closed point $\mathrm{Spec}(\mathbf{K})$ to $U \cap u_\mathbf{K}(\Phi_\mathbf{K})$, then $\zeta$ factors through the open subset $U \subseteq X$ and $\bar{\zeta}: \mathrm{Spec}(R) \to U \to \mathscr{U}$ extends $\mathrm{Spec}(K) \to \mathscr{V}$.
\end{enumerate}
\end{proof}
Furthermore, the condition (E) for the quotient stack $\mathscr{U}=[U/\mathbf{G}_m]$ implies certain absorption properties of the atlas.
\begin{lem}[\cite{MR755486}, Proposition 2.1 if $k=\mathbf{C}$]\label{A^0}
Let $X$ be a normal variety over a field $k$ with a $\mathbf{G}_m$-action. Let $U \subseteq X$ be a $\mathbf{G}_m$-invariant open subset such that the quotient stack $\mathscr{U}=[U/\mathbf{G}_m]$ satisfies the condition $\mathrm{(E)}$. For any connected component $X_i$ of $X^0$, we have $U \cap X_i \neq \emptyset$ implies that $X_i \subseteq U$.
\end{lem}
\begin{proof}
The connected component $X_i$ is irreducible by Remark \ref{0948-4} and $U \cap X_i \subseteq X_i$ is open, we show it is closed via valuative criterion for universal closedness (see \cite[\href{https://stacks.math.columbia.edu/tag/05JY}{Tag 05JY}]{stacks-project}). For every DVR $R$ with fraction field $K$ and any commutative diagram
\[
\begin{tikzcd}
\mathrm{Spec}(K) \ar[d,hook] \ar[r,"u_K"] & U \cap X_i \ar[d,hook] \\
\mathrm{Spec}(R) \ar[r,"u_R"'] & X_i
\end{tikzcd}
\]
of solid arrows, we have the following commutative diagram
\[
\begin{tikzcd}
\mathrm{Spec}(K) \ar[d,hook] \ar[r,"u_K"] & U \cap X_i \ar[r] & \mathscr{U} \ar[d] \\
\mathrm{Spec}(R) \ar[urr,dashed] \ar[rr] & & \mathrm{Spec}(k).
\end{tikzcd}
\]
Since $\mathscr{U}$ satisfies the condition (E) the dotted arrow exists, which can further lift to the atlas $u'_R: \mathrm{Spec}(R) \to U$ such that $u'_K=u_K$. Then $u_R=u'_R$ by separatedness of $X$ and $u_R$ factors through $U \cap X_i \subseteq X_i$.
\end{proof}
This can be generalized as follows.
\begin{lem}[\cite{MR709583}, Lemma 1.1.1 if $k=\mathbf{C}$]\label{new1}
Let $X$ be a normal variety over a field $k$ with a $\mathbf{G}_m$-action. Let $U \subseteq X$ be a $\mathbf{G}_m$-invariant open subset such that the quotient stack $\mathscr{U}=[U/\mathbf{G}_m]$ satisfies the condition $\mathrm{(E)}$. For any $\mathbf{G}_m$-invariant irreducible subset $\Sigma \subseteq X_i^- \cap X_j^+$, we have $U \cap \Sigma \neq \emptyset$ implies that $\Sigma \subseteq U$.
\end{lem}
\begin{rmk}\label{new-2}
There exist examples where $X_i^- \cap X_j^+$ is disconnected, even if $X$ is smooth projective, see \cite[\S 2]{MR704991}. In particular, Lemma \ref{new1} tells that if $U \subseteq X$ admits a proper good quotient, then
\[
U=\bigsqcup_{\Sigma \subseteq X_i^- \cap X_j^+ \text{ irred comp} \atop \text{s.t. } U \cap \Sigma \neq \emptyset} \Sigma,
\]
i.e., irreducible components of $X_i^- \cap X_j^+$ are building blocks for $\mathbf{G}_m$-invariant open subsets of $X$ with proper good quotients. However, Theorem \ref{onlyif} indicates that such open subsets are actually built out of $X_i^- \cap X_j^+$, instead of their irreducible components. It is still mysterious why they do not depend on the irreducible components.
\end{rmk}
\begin{proof}
If $U \cap X_i \neq \emptyset$ or $U \cap X_j \neq \emptyset$, then $X_i \subseteq U$ or $X_j \subseteq U$ by Lemma \ref{A^0}, and hence $X_i^- \subseteq U$ or $X_j^+ \subseteq U$. In either case we have $\Sigma \subseteq X_i^- \cap X_j^+ \subseteq U$. 

Suppose $U \cap X_i=\emptyset=U \cap X_j$. Similarly as before we show $U \cap \Sigma \subseteq \Sigma$ is closed. For every DVR $R$ with fraction field $K$ and residue field $\kappa$, and any commutative diagram
\[
\begin{tikzcd}
\mathrm{Spec}(K) \ar[d,hook] \ar[r,"u_K"] & U \cap \Sigma \ar[d,hook] \\
\mathrm{Spec}(R) \ar[r,"u_R"'] & \Sigma
\end{tikzcd}
\]
of solid arrows, as in the proof of Lemma \ref{A^0} we have a morphism $u'_R: \mathrm{Spec}(R) \to U$ such that $u'_K=\eta_K.u_K$ for some $\eta_K \in \mathbf{G}_m(K) \subseteq \mathbf{P}^1(K)$. Let $\eta_R \in \mathbf{P}^1(R)$ be the element extending $\eta_K$, then the complete orbit map of $u_R$ locates in
\[
\overline{\sigma}_{u_R}: \mathbf{P}^1_R \to X_i \cup \Sigma \cup X_j
\]
and the composition
\[
u''_R: \mathrm{Spec}(R) \xrightarrow{(\eta_R,\mathrm{id})} \mathbf{P}^1_R \xrightarrow{\overline{\sigma}_{u_R}} X_i \cup \Sigma \cup X_j
\]
coincides with $u'_R$ on $\mathrm{Spec}(K)$, so $u''_R=u'_R$ by separatedness of $X$. This means $u''_R$ factors through $U \cap (X_i \cup \Sigma \cup X_j)=U \cap \Sigma$ and hence $u''_\kappa=\eta_\kappa.u_\kappa \in U \cap \Sigma$. If $\eta_\kappa \in (\mathbf{P}^1-\mathbf{G}_m)(\kappa)$, then $U \cap X_i \neq \emptyset$ or $U \cap X_j \neq \emptyset$, a contradiction. Therefore $\eta_\kappa \in \mathbf{G}_m(\kappa)$ and $u_\kappa \in U \cap \Sigma$, which implies that $u_R$ factors through $U \cap \Sigma \subseteq \Sigma$.
\end{proof}
\begin{cor}[\cite{MR755486}, Lemma 2.6 if $k=\mathbf{C}$]\label{(1,r)-in-U}
Let $X$ be a normal proper variety over a field $k$ with a $\mathbf{G}_m$-action. Let $U \subseteq X$ be a $\mathbf{G}_m$-invariant open subset such that the quotient stack $\mathscr{U}=[U/\mathbf{G}_m]$ satisfies the condition $\mathrm{(E)}$. Then $X_{\min}^- \cap X_{\max}^+ \subseteq U$. In particular, every smoothable maximal chain of orbits in $X$ is $U$-smoothable.
\end{cor}
\begin{proof}
By Lemma \ref{source-sink} the subset $X_{\min}^- \cap X_{\max}^+ \subseteq X$ is open (hence irreducible) dense, then $U \cap (X_{\min}^- \cap X_{\max}^+) \neq \emptyset$ and Lemma \ref{new1} applies.
\end{proof}
\subsection{Consequence of (S)}
The (S)-part of \textbf{Expectation} holds.
\begin{prop}\label{f2}
Let $X$ be a proper variety over a field $k$ with a $\mathbf{G}_m$-action. For any $\mathbf{G}_m$-invariant open subset $U \subseteq X$, the following are equivalent:
\begin{enumerate}
\item
The quotient stack $\mathscr{U}=[U/\mathbf{G}_m]$ is S-complete.
\item
The image of any smoothable maximal chain of orbits in $X$ intersects $U$ with one of following forms (see Figure \ref{poss-fig}):
\begin{enumerate}
\item
$\emptyset$;
\item
$\mathbf{G}_m.x$ for some $x \in X-X^0$;
\item
$\mathbf{G}_m.x_1 \cup \{x_1^+=x_2^-\} \cup \mathbf{G}_m.x_2$ for some $x_1,x_2 \in X-X^0$;
\item
$\mathbf{G}_m.x \cup \{x^-\}$ (resp., $\mathbf{G}_m.x \cup \{x^+\}$) for some $x \in X_{\min}^--X_{\min}$ (resp., $x \in X_{\max}^+-X_{\max}$).
\end{enumerate}
\end{enumerate}
\end{prop}
\begin{proof}
By Theorem \ref{thm:S-com-new}, the stack $\mathscr{U}$ is S-complete if and only if 
\[
(1') \quad \mathscr{U} \text{ is S-complete relative to } \mathscr{V} \subseteq \mathscr{U}.
\]
\begin{enumerate}
\item[$(1') \Rightarrow (2)$.]
Let $f_{\mathbf{K}}: C_{\mathbf{K}} \to X$ be a smoothable maximal chain of orbits in $X$. Notice that $U \cap f_{\mathbf{K}}(C_{\mathbf{K}}) \subseteq f_{\mathbf{K}}(C_{\mathbf{K}})$ is $\mathbf{G}_m$-invariant and open, it suffices to show the closures of any two $\mathbf{G}_m$-orbits in $U \cap f_{\mathbf{K}}(C_{\mathbf{K}})$ intersect in $U \cap f_{\mathbf{K}}(C_{\mathbf{K}})$, i.e., if $\mathbf{G}_m.x_1 \neq \mathbf{G}_m.x_2 \subseteq U \cap f_{\mathbf{K}}(C_{\mathbf{K}})$, then $x_1^+=x_2^- \in U$ or $x_1^-=x_2^+ \in U$. Suppose $U \cap f_{\mathbf{K}}(C_{\mathbf{K}}) \neq \emptyset$.

By definition there exist a DVR $R$ with fraction field $K$ and residue field $\mathbf{K}$, and a $\mathbf{G}_m$-equivariant commutative diagram \eqref{1310-2} such that the generic fiber $f_K: C_K \to X$ is a complete orbit map of some point $x_K \in V(K)$ (since $U \cap f_{\mathbf{K}}(C_{\mathbf{K}}) \neq \emptyset$). Choose a section $s_i: \mathrm{Spec}(R) \to C_R$ for $i=1,2$ such that the composition 
\[
\Gamma_i: \mathrm{Spec}(R) \xrightarrow{s_i} C_R \xrightarrow{f_R} X
\]
maps the closed point $\mathrm{Spec}(\mathbf{K})$ to $x_i$. The morphism $\Gamma_i$ thus factors through the open subset $U \subseteq X$. Since $f_K$ is the complete orbit map of $x_K$, we have $\mathbf{G}_{m}.\Gamma_1|_K=\mathbf{G}_{m}.\Gamma_2|_K \subseteq U$. The pair $(\Gamma_1,\Gamma_2)$ then defines a morphism $\mathrm{Spec}(R) \cup_{\mathrm{Spec}(K)} \mathrm{Spec}(R)=\overline{\mathrm{ST}}_R-\{0\} \to \mathscr{U}$ such that $\mathrm{Spec}(K) \hookrightarrow \overline{\mathrm{ST}}_R-\{0\} \to \mathscr{U}$ factors through the open dense substack $\mathscr{V} \subseteq \mathscr{U}$. By hypothesis it extends to a morphism $\Gamma: \overline{\mathrm{ST}}_R \to \mathscr{U}$ and the condition $\Gamma(0) \in \mathscr{U}$ is exactly what we want.
\item[$(2) \Rightarrow (1')$.]
For every DVR $R$ with fraction field $K$ and residue $\mathbf{K}$, and any commutative diagram
\[
\begin{tikzcd}
\overline{\mathrm{ST}}_R-\{0\} \ar[r,"u"] \ar[d,hook] & \mathscr{U} \ar[d] \\
\overline{\mathrm{ST}}_R \ar[ur,dashed] \ar[r] & \mathrm{Spec}(k)
\end{tikzcd}
\]
of solid arrows such that $\mathrm{Spec}(K) \hookrightarrow \overline{\mathrm{ST}}_R-\{0\} \to \mathscr{U}$ factors through the open dense substack $\mathscr{V} \subseteq \mathscr{U}$, by the proof of \cite[Proposition 2.9]{MR3758902} there exists a $\lambda$-equivariant lifting $u'$ of $u$, for some cocharacter $\lambda \in X_*(\mathbf{G}_m)$. Consider the following commutative diagram
\[
\begin{tikzcd}
\mathrm{Bl}_{\mathscr{I}}(\mathrm{Bl}_z(\mathbf{P}^1_R)) \ar[r,"\tilde{u}_R"] \ar[d] & X \\
\mathrm{Bl}_z(\mathbf{P}^1_R) \ar[ur,"u'"',dashed] \\
\mathrm{Spec}(R[x,y]/xy-\pi) \ar[u,hook] \\
\mathrm{Spec}(R[x,y]/xy-\pi)-\{0\} \ar[r,"u'"] \ar[d] \ar[u,hook] & U \ar[d] \ar[uuu,hook] \\
\overline{\mathrm{ST}}_R-\{0\} \ar[r,"u"'] & \mathscr{U}
\end{tikzcd}
\]
where $z=(\mathrm{Spec}(\mathbf{K}),\infty) \in \mathbf{P}^1_R$ and $\mathscr{I} \subseteq \mathscr{O}_{\mathrm{Bl}_z(\mathbf{P}^1_R)}$ is some ideal supported at $\mathrm{Spec}(\mathbf{K})$. As in the proof of Proposition \ref{cover}, the special fiber $\tilde{u}_{\mathbf{K}}: \Phi_\mathbf{K} \to X$ of $\tilde{u}_R$ can be refined as a chain of orbits, which is smoothable and maximal (since $\mathrm{Spec}(K) \hookrightarrow \overline{\mathrm{ST}}_R-\{0\} \to \mathscr{U}$ factors through the open dense substack $\mathscr{V} \subseteq \mathscr{U}$). Then by hypothesis $U \cap \tilde{u}_{\mathbf{K}}(\Phi_{\mathbf{K}}) \neq \emptyset$ is of the form (b), (c) or (d). Choose a section $s: \mathrm{Spec}(R[x,y]/xy-\pi) \to \mathrm{Bl}_{\mathscr{I}}(\mathrm{Bl}_z(\mathbf{P}^1_R))$ such that the composition
\[
\zeta: \mathrm{Spec}(R[x,y]/xy-\pi) \xrightarrow{s} \mathrm{Bl}_{\mathscr{I}}(\mathrm{Bl}_z(\mathbf{P}^1_R)) \xrightarrow{\tilde{u}_R} X
\]
maps the closed point $\mathrm{Spec}(\mathbf{K})$ to $U \cap \tilde{u}_{\mathbf{K}}(\Phi_{\mathbf{K}})$, then $\zeta$ factors through the open subset $U \subseteq X$ and descends to a morphism $\overline{\mathrm{ST}}_R \to \mathscr{U}$, which extends $u$ by construction. The uniqueness follows since in either case there is a unique closed point in $U \cap \tilde{u}_{\mathbf{K}}(\Phi_{\mathbf{K}})$.
\end{enumerate}
\end{proof}
Since $X$ is covered by smoothable maximal chain of orbits, the configurations in Proposition \ref{f2} has several quick consequences.
\begin{cor}[\cite{MR755486}, Proposition 2.2. if $k=\mathbf{C}$]\label{hhh}
Let $X$ be a proper variety over a field $k$ with a $\mathbf{G}_m$-action. Let $U \subseteq X$ be a $\mathbf{G}_m$-invariant open subset such that the quotient stack $\mathscr{U}=[U/\mathbf{G}_m]$ is S-complete. If $X_i \subseteq U$, then $X_j \nsubseteq U$ for any $j \neq i$ with $X_i <_d X_j$ or $X_j <_d X_i$. In particular, there is no complete orbit in $U$.
\end{cor}
\begin{proof}
If $X_j \subseteq U$, then the intersection of $U$ with the image of any smoothable maximal chain of orbits in $X$ passing through $X_i^- \cap X_j^+$ or $X_j^- \cap X_i^+$ contains a complete orbit, a contradiction to Proposition \ref{f2}.
\end{proof}
\begin{cor}[\cite{MR709583}, Theorem 1.4 if $k=\mathbf{C}$ and $U^0=\emptyset$]\label{1727}
Let $X$ be a proper variety over a field $k$ with a $\mathbf{G}_m$-action. Let $U \subseteq X$ be a $\mathbf{G}_m$-invariant open subset and $Z:=X-U$ be its closed complement. If the quotient stack $\mathscr{U}=[U/\mathbf{G}_m]$ is S-complete, then the following are equivalent:
\begin{enumerate}
\item
The complement $Z$ is disconnected.
\item
The complement $Z$ has two connected components.
\end{enumerate}
In this case, $X_{\min}$ and $X_{\max}$ are in the different connected components of $Z$.
\end{cor}
\begin{proof}
For any point $x \in Z$, pick a smoothable maximal chain of orbits in $X$ passing through it, then $x$ is in the same connected component of $Z$ as $X_{\min}$ or $X_{\max}$ by Proposition \ref{f2}. This shows that $Z$ has at most two connected components.
\end{proof}
Coupled with Corollary \ref{(1,r)-in-U}, the conclusions in Proposition \ref{f1} and \ref{f2} can be summarized as follows, which is precisely \textbf{Expectation}.
\begin{thm}[\cite{MR755486}, Lemma 2.7 if $k=\mathbf{C}$]\label{summary}
Let $X$ be a normal proper variety over a field $k$ with a $\mathbf{G}_m$-action. For any $\mathbf{G}_m$-invariant open subset $U \subseteq X$, the quotient stack $\mathscr{U}=[U/\mathbf{G}_m]$ satisfies the conditions $\mathrm{(S)}$ and $\mathrm{(E)}$ if and only if the image of any smoothable maximal chain of orbits in $X$ intersects $U$ with one of the following forms (see Figure \ref{poss-fig}):
\begin{enumerate}
\item
$\mathbf{G}_m.x$ for some $x \in X-X^0$;
\item
$\mathbf{G}_m.x_1 \cup \{x_1^+=x_2^-\} \cup \mathbf{G}_m.x_2$ for some $x_1,x_2 \in X-X^0$;
\item
$\mathbf{G}_m.x \cup \{x^-\}$ (resp., $\mathbf{G}_m.x \cup \{x^+\}$) for some $x \in X_{\min}^--X_{\min}$ (resp., $x \in X_{\max}^+-X_{\max}$).
\end{enumerate}
\end{thm}
\begin{rmk}\label{rmk:excep}
The form (3) in Theorem \ref{summary} happens if and only if $U \cap X_{\min} \neq \emptyset$ (resp., $U \cap X_{\max} \neq \emptyset$), in which case $X_{\min}^- \subseteq U$ (resp., $X_{\max}^+ \subseteq U$) by Lemma \ref{A^0}. Since $X_{\min}^- \subseteq X$ (resp., $X_{\max}^+ \subseteq X$) is a $\mathbf{G}_m$-invariant open subset and admits a proper good quotient $X_{\min}$ (resp., $X_{\max}$), it is reasonable to expect that $U=X_{\min}^-$ (resp., $U=X_{\max}^+$). This follows, for example, if $X_{\min}^- \subseteq U$ (resp., $X_{\max}^+ \subseteq U$) is closed. This is exactly the implication of the last condition: $\Theta$-reductivity.
\end{rmk}
\subsection{Consequence of ($\Theta$)}
Denote by 
\begin{itemize}
\item
$\mathrm{B}\mathbf{G}_m:=[\mathrm{Spec}(\mathbf{Z})/\mathbf{G}_m]$ the classifying stack of $\mathbf{G}_m$.
\item
$\Theta:=[\mathbf{A}^1/\mathbf{G}_m]$ the quotient stack defined by the standard contracting action of $\mathbf{G}_m$ on the affine line $\mathbf{A}^1=\mathrm{Spec}(\mathbf{Z}[t])$.
\end{itemize}
Both stacks are defined over $\mathrm{Spec}(\mathbf{Z})$ and therefore pull back to any base. For any DVR $R$ with residue field $\kappa$, let $\Theta_R:=\Theta \times \mathrm{Spec}(R)$ and $0:=[0/\mathbf{G}_m] \times \mathrm{Spec}(\kappa)$ be its unique closed point.
\begin{defn}[\cite{MR4665776}, Definition 3.10]

A morphism $f: \mathscr{X} \to \mathscr{Y}$ of locally noetherian algebraic stacks is $\Theta$\emph{-reductive} if for every DVR $R$, any commutative diagram
\[
\begin{tikzcd}
\Theta_R-\{0\} \ar[r] \ar[d,hook] & \mathscr{X} \ar[d,"f"] \\
\Theta_R \ar[ur,dashed,"{\exists !}"'] \ar[r] & \mathscr{Y}
\end{tikzcd}
\]
of solid arrows can be uniquely filled in.
\end{defn}
$\Theta$-reductivity of quotient stacks can be characterized in terms of the evaluation maps to the atlas (see \cite[Proposition 3.13]{MR4665776}). In our setup it has the following form.
\begin{lem}\label{theta-red}
Let $X$ be a variety over a field $k$ with a $\mathbf{G}_m$-action. Let $U \subseteq X$ be a $\mathbf{G}_m$-invariant open subset. The quotient stack $\mathscr{U}=[U/\mathbf{G}_m]$ is $\Theta$-reductive if and only if the morphism $\mathrm{ev}_1: (U \cap X_i)^\pm \to U$ is a closed immersion for each $i \in \pi_0(X^0)$. In particular, if $U$ has no fixed points, then $\mathscr{U}$ is $\Theta$-reductive.
\end{lem}
\begin{proof}
By \cite[Proposition 3.13 and Remark 3.14]{MR4665776} the quotient stack $\mathscr{U}$ is $\Theta$-reductive if and only if the evaluation morphism $\mathrm{ev}_1: U^+_\lambda \to U$ is a closed immersion for any cocharacter $\lambda \in X_*(\mathbf{G}_m)$, where $U^+_\lambda$ is the scheme representing the functor $\underline{\mathrm{Map}}_k^{\lambda}(\mathbf{A}^1,U)$ of $\lambda$-equivariant morphisms. If $n_\lambda \in \mathbf{Z}$ is the integer corresponding to $\lambda$, then by \cite[Proposition 5.24 (1)]{MR4088350} we have
\[
U_{\lambda}^0=
\begin{cases}
U^0 & \text{if } n_\lambda \neq 0 \\
U & \text{if } n_\lambda=0
\end{cases}
\text{ and }
U_{\lambda}^+=
\begin{cases}
U^+ & \text{if } n_\lambda > 0 \\
U & \text{if } n_\lambda=0 \\
U^- & \text{if } n_\lambda < 0
\end{cases}.
\]
Thus $\mathscr{U}$ is $\Theta$-reductive if and only if $\mathrm{ev}_1: U^{\pm} \to U$ is a closed immersion. Finally applying \cite[Lemma 1.4.7]{Drinfeld-Gm} twice yields that
\begin{align*}
    U^{\pm}&=\mathrm{ev}_0^{-1}(U^0)=\mathrm{ev}_0^{-1}(U \cap X^0)=\mathrm{ev}_0^{-1}\left(\bigsqcup_{i \in \pi_0(X^0)} U \cap X_i\right) \\
    &=\bigsqcup_{i \in \pi_0(X^0)} \mathrm{ev}_0^{-1}(U \cap X_i) \\
    &=\bigsqcup_{i \in \pi_0(X^0)} (U \cap X_i)^\pm,
\end{align*}
and we are done.
\end{proof}
\subsection{A topological characterization}
As highlighted in Remark \ref{rmk:excep}, if for an open subset $U \subseteq X$ the form (3) in Theorem \ref{summary} does not happen, then 
\[
U \cap X_{\min}=\emptyset=U \cap X_{\max}.
\]
In this case $X_{\min}$ and $X_{\max}$ are in the complement $Z:=X-U$, and they cannot be connected via smoothable maximal chains of orbits by Theorem \ref{summary}. Hence, it is highly probable that the complement $Z$ is disconnected. Notably, this characterizes the properness of the good quotient of $U$.
\begin{prop}[\cite{MR709583}, Theorem 1.4 if $k=\mathbf{C}$]\label{prop:2-connected}
Let $X$ be a normal proper variety over a field $k$ with a $\mathbf{G}_m$-action. Let $U \subseteq X - X^0$ be a $\mathbf{G}_m$-invariant open dense subset with a separated good quotient $U/\mathbf{G}_m$. The following are equivalent: 
\begin{enumerate}
\item
The separated good quotient $U/\mathbf{G}_m$ is proper.
\item
The closed complement $Z:=X-U$ has two connected components.
\end{enumerate}
In this case, $X_{\min}$ and $X_{\max}$ are in the different connected components of $Z$.
\end{prop}
First we deal with the relatively easy case that $X$ is smooth and then explain how to modify the arguments to work in general.
\subsubsection{Smooth case}
Our proof in this case is motivated by \cite[Theorem 1.3]{MR784002}. Hereafter, cohomology is either
\begin{enumerate}
\item
singular cohomology with coefficient $\mathbf{Q}$ if $k \subseteq \mathbf{C}$, or
\item
\'{e}tale cohomology with coefficient $\mathbf{Q}_\ell$ with $\ell \neq \mathrm{char}(k)$ if $k \nsubseteq \mathbf{C}$.
\end{enumerate}
Denote by $\mathbf{K}$ the coefficient in both cases. Since $U/\mathbf{G}_m$ is proper (resp., $Z$ has two connected components) if and only if $H^0_c(U/\mathbf{G}_m,\mathbf{K})=\mathbf{K}$ (resp., $H^0(Z,\mathbf{K})=\mathbf{K}^2$), it is equivalent to show
\[
H^0_c(U/\mathbf{G}_m,\mathbf{K})=\mathbf{K} \text{ if and only if } H^0(Z,\mathbf{K})=\mathbf{K}^2.
\]
Let $\jmath: U \hookrightarrow X$ (resp., $\imath: Z \hookrightarrow X$) be the open (resp., closed) immersion. Applying the cohomology functor $H_c^*(X,-)$ to the exact triangle $\mathbf{R}\jmath_!\mathbf{K} \to \mathbf{K} \to \imath_*\mathbf{K} \xrightarrow{+1}$ gives rise to a long exact sequence
\[
\begin{tikzcd}
H_c^0(U,\mathbf{K}) \ar[r] \ar[d,equal] & H^0(X,\mathbf{K}) \ar[r] \ar[d,equal] & H^0(Z,\mathbf{K}) \ar[r] \ar[d,equal] & H_c^1(U,\mathbf{K}) \ar[r,"\jmath_*"] \ar[d,equal,"(\ast)"'] & H^1(X,\mathbf{K}) \ar[d,equal] \\
0 \ar[r] & \mathbf{K} \ar[r] & H^0(Z,\mathbf{K}) \ar[r] & H^0_c(U/\mathbf{G}_m,\mathbf{K}) \ar[r,"0","(\ast\ast)"'] & H^1(X,\mathbf{K})
\end{tikzcd}
\]
and we claim that it equals to the second row. The only non-trivial parts are $(\ast)$ and $(\ast\ast)$. Indeed, $(\ast)$ is a Leray spectral sequence computation and $(\ast\ast)$ is via comparing appropriate weights on both sides.
\begin{claim}\label{proper}
The composition $\lambda: U \xrightarrow{f} \mathscr{U} \xrightarrow{g} U/\mathbf{G}_m$, where $f$ is the $\mathbf{G}_m$-torsor and $g$ is the good moduli space, induces an isomorphism
\[
\lambda^*: H^0_c(U/\mathbf{G}_m,\mathbf{K}) \xrightarrow{\sim} H_c^1(U,\mathbf{K}).
\]
\end{claim}
\begin{proof}
The Leray spectral sequence for $\lambda$ reads
\[
E_2^{p,q}:=H_c^p(U/\mathbf{G}_m,\mathbf{R}^q\lambda_!\mathbf{K}) \Rightarrow H_c^{p+q}(U,\mathbf{K}).
\]
Since $\mathbf{R}^pg_! \circ \mathbf{R}^qf_! \Rightarrow \mathbf{R}^{p+q}\lambda_!$, we will compute $\mathbf{R}^i\lambda_!$ step-by-step.
\begin{enumerate}
\item
Firstly we have
\begin{equation}\label{sm1}
\mathbf{R}^if_!\mathbf{K} \cong \begin{cases}
0 & \text{if } i=0 \\
\mathbf{K} & \text{if } i=1,2 \\
0 & \text{if } i>2
\end{cases}.
\end{equation}
Indeed, consider the following diagram
\[
\begin{tikzcd}
U \ar[d,"f"'] \ar[r,hook,"j"] & {[U \times \mathbf{A}^1/\mathbf{G}_m]} \ar[dl,"\bar{f}"'] \\
\mathscr{U} \ar[bend right,ur,"s"']
\end{tikzcd}
\]
where $\bar{f}$ is the line bundle associated to the $\mathbf{G}_m$-torsor $f$, $s$ is its zero section and $j$ is the open immersion given by $x \mapsto (x,1)$. Applying $\mathbf{R}\bar{f}_!$ to the exact triangle $\mathbf{R}j_!\mathbf{K} \to \mathbf{K} \to s_*\mathbf{K} \xrightarrow{+1}$ yields
\begin{equation}\label{1931-3}
\mathbf{R}f_!\mathbf{K} \to \mathbf{R}\bar{f}_!\mathbf{K} \to \mathbf{K} \xrightarrow{+1}
\end{equation}
To conclude we claim that there is a quasi-isomorphism $\mathbf{R}\bar{f}_!\mathbf{K} \simeq \mathbf{K}[-2]$ and thus we obtain \eqref{sm1} by looking at the long exact sequence in cohomology associated to the exact triangle \eqref{1931-3}. Indeed, since the fibres of $\bar{f}$ are $\mathbf{A}^1$, which are $\mathbf{K}$-acyclic, the morphism $\mathbf{R}\bar{f}_*\mathbf{K} \to \mathbf{K}$ (applying $\mathbf{R}\bar{f}_*$ to the exact triangle $\mathbf{R}j_!\mathbf{K} \to \mathbf{K} \to s_*\mathbf{K} \xrightarrow{+1}$) is a quasi-isomorphism and Poincar\'{e} duality tells us that $\mathbf{R}\bar{f}_! \cong \mathbf{R}\bar{f}_*[-2]$.
\item
Secondly we have
\begin{equation}\label{sm2}
\mathbf{R}^ig_!\mathbf{K} \cong \begin{cases}
\mathbf{K} & \text{if } i=0 \\
0 & \text{if } i \neq 0
\end{cases}.
\end{equation}
Indeed, since $U$ has no fixed points, for any point $\bar{x} \in U/\mathbf{G}_m$, the fiber $g^{-1}(\bar{x})=[\mathbf{G}_m.x/\mathbf{G}_m]=\mathrm{B}I_x$ where $x \in \lambda^{-1}(\bar{x})$ is any point and $I_x \subseteq \mathbf{G}_m$ is its stabilizer group, which is finite. Therefore
\[
(\mathbf{R}^ig_!\mathbf{K})_{\bar{x}}=H^{i}_c(\mathrm{B}I_x,\mathbf{K})=H^{-i}(\mathrm{B}I_x,\mathbf{K})=\begin{cases}
\mathbf{K} & \text{if } i=0 \\
0 & \text{if } i \neq 0
\end{cases}.
\]
This shows $\mathbf{R}^ig_!\mathbf{K}=0$ for $i \neq 0$. As for $i=0$, the natural morphism $g_!\mathbf{K} \to g_*\mathbf{K} \cong \mathbf{K}$ induces isomorphisms on stalks, witnessing $g_!\mathbf{K} \cong \mathbf{K}$.
\end{enumerate}
Altogether, this shows
\begin{equation}\label{1150-4}
\mathbf{R}^i\lambda_!\mathbf{K}=g_! \circ \mathbf{R}^if_!\mathbf{K}=\begin{cases}
\mathbf{K} & \text{if } i=1,2 \\
0 & \text{if } i \neq 1,2
\end{cases},
\end{equation}
and hence $H_c^1(U,\mathbf{K})=H_c^0(U/\mathbf{G}_m,\mathbf{R}^1\lambda_!\mathbf{K})=H_c^0(U/\mathbf{G}_m,\mathbf{K})$.
\end{proof}
\begin{claim}[\cite{MR784002}, Theorem 1.5 if $k=\mathbf{C}$]\label{zero}
The map $\jmath_*: H_c^1(U,\mathbf{K}) \to H^1(X,\mathbf{K})$ is a zero map.
\end{claim}
\begin{proof}
Both sides carry weights and the map $\jmath_*$ is weight-preserving. This claim is proved by comparing them.
\begin{enumerate}
\item
The cohomology $H_c^1(U,\mathbf{K})=H_c^0(U/\mathbf{G}_m,\mathbf{K})$ is pure of weight $0$.

This is Deligne's Hodge III \cite[Th\'{e}or\`{e}me 8.2.4]{MR498552} (resp., Deligne's Weil II \cite[Corollaire 3.3.3]{MR601520}) if $k \subseteq \mathbf{C}$ (resp., $k \nsubseteq \mathbf{C}$).
\item
The cohomology $H^1(X,\mathbf{K})$ is pure of weight $1$. This is the only place we use the smoothness of $X$.

This is Deligne's Hodge II \cite[Corollaire 3.2.15]{MR498551} (resp., Deligne's Weil II \cite[Corollaire 3.3.6]{MR601520}) if $k \subseteq \mathbf{C}$ (resp., $k \nsubseteq \mathbf{C}$).
\end{enumerate}
Then the map $\jmath_*$ must vanish.
\end{proof}
\subsubsection{Normal case}
The only place where the smoothness assumption on $X$ is used is the purity of its first cohomology. To retain it in the singular case we replace the constant sheaf $\mathbf{K}$ by the intersection complex $\mathrm{IC}_X$ (with respect to the middle perversity) with coefficient $\mathbf{K}$. This replacement leads to two issues to be settled:
\begin{enumerate}
\item
the intersection cohomology coincides with the usual cohomology in degree $0$ and
\item
an analogy of Claim \ref{proper} holds.
\end{enumerate}
The first issue is solved in the following lemma, which will be used frequently in the computation of hypercohomology spectral sequences. This is also the major extra input in the singular case, as a price of replacing sheaves by complexes.
\begin{lem}\label{1918}
Let $V$ be an $n$-dimensional normal variety over a field $k$. Then $\mathscr{H}^{-n}(\mathrm{IC}_V)=\mathbf{K}$. In particular, there exists an exact triangle $\mathbf{K}[n] \to \mathrm{IC}_V \to \tau_{\geq 1}\mathrm{IC}_X \xrightarrow{+1}$ and $\mathrm{IH}_c^0(V)=H_c^0(V,\mathbf{K})$.
\end{lem}
\begin{proof}
Let $j: V_{sm} \hookrightarrow V$ be the open immersion of the smooth locus. Since $V$ is normal, the complement of $V_{sm}$ in $V$ has codimension at least $2$. The strong support condition reads
\[
\mathscr{H}^{<-n}(\mathrm{IC}_V)=0 \text{ and } \mathscr{H}^{-n}(\mathrm{IC}_V)=j_*\mathbf{K}=\mathbf{K}.
\]
In particular, the standard exact triangle
\[
\tau_{\leq 0}(\mathrm{IC}_V[-n]) \to \mathrm{IC}_V[-n] \to \tau_{\geq 1}(\mathrm{IC}_V[-n]) \xrightarrow{+1}
\]
becomes $\mathbf{K} \to \mathrm{IC}_V[-n] \to (\tau_{\geq -n+1}\mathrm{IC}_V)[-n] \xrightarrow{+1}$, i.e.,
\[
\mathbf{K}[n] \to \mathrm{IC}_V \to \tau_{\geq -n+1}\mathrm{IC}_V \xrightarrow{+1}
\]
and
\[
\mathrm{IH}_c^0(V):=\mathbf{H}^{-n}_c(V,\mathrm{IC}_V)=H^0_c(V,\mathscr{H}^{-n}(\mathrm{IC}_V))=H^0_c(V,\mathbf{K})
\]
where the first equality follows from the hypercohomology spectral sequence 
\[
E_2^{p,q}:=H_c^p(V,\mathscr{H}^q(\mathrm{IC}_V)) \Rightarrow \mathbf{H}_c^{p+q}(V,\mathrm{IC}_V)
\]
and the fact $H_c^p(V,-) \neq 0$ only for $p \geq 0$ and $\mathscr{H}^{q}(\mathrm{IC}_V) \neq 0$ only for $q \geq -n$.
\end{proof}
Similarly, applying the hypercohomology functor $\mathbf{H}_c^{*-n}(X,-)$ to the exact triangle $\jmath_!\jmath^*\mathrm{IC}_X \to \mathrm{IC}_X \to \imath_*\imath^*\mathrm{IC}_X \xrightarrow{+1}$ gives rise to a long exact sequence
\[
\begin{tikzcd}
\mathrm{IH}_c^0(U) \ar[r] \ar[d,equal,"\text{Lem } \ref{1918}"'] & \mathrm{IH}^0(X) \ar[r] \ar[d,equal,"\text{Lem } \ref{1918}"'] & \mathbf{H}^{-n}(Z,\imath^*\mathrm{IC}_X) \ar[r] \ar[d,equal,"\text{Lem } \ref{1918}"',"\text{hypercohomology s.s}"] & \mathrm{IH}_c^1(U) \ar[r,"\jmath_*"] \ar[d,equal] & \mathrm{IH}^1(X) \ar[d,equal] \\
H_c^0(U,\mathbf{K}) \ar[r] \ar[d,equal] & H^0(X,\mathbf{K}) \ar[r] \ar[d,equal] & H^0(Z,\mathbf{K}) \ar[r] \ar[d,equal] & \mathrm{IH}_c^1(U) \ar[r] \ar[d,equal,"(*)"'] & \mathrm{IH}^1(X) \ar[d,equal] \\
0 \ar[r] & \mathbf{K} \ar[r] & H^0(Z,\mathbf{K}) \ar[r] & H^0_c(U/\mathbf{G}_{m},\mathbf{K}) \ar[r,"0","(**)"'] & \mathrm{IH}^1(X)
\end{tikzcd}
\]
and we claim it equals to the third row. Note that in this case we get $(**)$ for free by the nature of our replacement: $\mathrm{IH}_c^1(X)$ is pure of weight $1$, using Gabber's purity theorem \cite[Corollary 11.3.5]{MR4269941} (resp., \cite[Corollaire 5.3.2 and Th\'{e}or\`{e}me 5.4.1]{MR751966}) if $k \subseteq \mathbf{C}$ (resp., $k \nsubseteq \mathbf{C}$). Therefore it remains to prove $(*)$.
\begin{claim}
The composition $\lambda: U \xrightarrow{f} \mathscr{U} \xrightarrow{g} U/\mathbf{G}_m$ induces an isomorphism
\[
\lambda^*: H^0_c(U/\mathbf{G}_m,\mathbf{K}) \xrightarrow{\sim} \mathrm{IH}_c^1(U).
\]
\end{claim}
\begin{proof}
The Leray spectral sequence for $\lambda$ reads
\[
E_2^{p,q}:=H_c^p(U/\mathbf{G}_m,\mathbf{R}^q\lambda_!\mathrm{IC}_U) \Rightarrow \mathbf{H}_c^{p+q}(U,\mathrm{IC}_U).
\]
Since $\mathbf{R}^pg_! \circ \mathbf{R}^qf_! \Rightarrow \mathbf{R}^{p+q}\lambda_!$, we will compute $\mathbf{R}^i\lambda_!$ step-by-step.
\begin{enumerate}
\item
Firstly we have
\[
\mathbf{R}f_!\mathrm{IC}_U \cong (\mathrm{IC}_{\mathscr{U}} \otimes^{\mathbf{L}} \mathbf{R}f_!\mathbf{K})[1]
\]
Since $f$ is smooth, we could apply \cite[Lemma 6.1]{MR2480756} (or \cite[5.4.2 Theorem]{MR696691}) to obtain that
\[
f^!\mathrm{IC}_{\mathscr{U}} \cong \mathrm{IC}_U[-1]
\]
Applying $\mathbf{R}f_!$ and using projection formula we have
\[
\mathbf{R}f_!\mathrm{IC}_U[-1] \cong \mathbf{R}f_!(f^!\mathrm{IC}_{\mathscr{U}}) \cong \mathbf{R}f_!(f^!\mathrm{IC}_{\mathscr{U}} \otimes^{\mathbf{L}} \mathbf{K}) \cong \mathrm{IC}_{\mathscr{U}} \otimes^{\mathbf{L}} \mathbf{R}f_!\mathbf{K}
\]
i.e., $\mathbf{R}f_!\mathrm{IC}_U=(\mathrm{IC}_{\mathscr{U}} \otimes^{\mathbf{L}} \mathbf{R}f_!\mathbf{K})[1]$.
\item
Secondly we have
\begin{equation}\label{1150-5}
\mathbf{R}^ig_!\mathrm{IC}_{\mathscr{U}} \cong \begin{cases}
0 & \text{if } i<1-n \\
\mathbf{K} & \text{if } i=1-n
\end{cases}
\end{equation}
Indeed, for any point $\bar{x} \in U/\mathbf{G}_m$
\begin{align*}
(\mathbf{R}^ig_!\mathrm{IC}_{\mathscr{U}})_{\bar{x}}&=\mathbf{H}^{i}_c(\mathrm{B}I_x,\iota_x^*\mathrm{IC}_{\mathscr{U}})=H_c^0(\mathrm{B}I_x,\mathscr{H}^i(\iota_x^*\mathrm{IC}_{\mathscr{U}}))=H^0(\mathrm{B}I_x,\iota_x^*\mathscr{H}^i(\mathrm{IC}_{\mathscr{U}})) \\
& \cong \begin{cases}
0 & \text{if } i < 1-n \\
\mathbf{K} & \text{if } i=1-n
\end{cases}
\end{align*}
where $\iota_x: \mathrm{B}I_x \hookrightarrow \mathscr{U}$ is the residue gerbe of $\mathscr{U}$ at $x$. This shows that $\mathbf{R}^ig_!\mathrm{IC}_{\mathscr{U}}=0$ for $i < 1-n$. As for $i=1-n$, there exists a morphism $\mathbf{K}[n-1] \to \mathrm{IC}_{\mathscr{U}}$ (see the proof of Lemma \ref{1918}), using \eqref{sm2} we obtain a morphism $\mathbf{K} \to \mathbf{R}^{1-n}g_!\mathrm{IC}_{\mathscr{U}}$ inducing isomorphisms on stalks, witnessing $\mathbf{R}^{1-n}g_!\mathrm{IC}_{\mathscr{U}} \simeq \mathbf{K}$.
\end{enumerate}
Altogether, this shows
\[
\mathbf{R}\lambda_!\mathrm{IC}_U=\mathbf{R}g_! \circ \mathbf{R}f_!\mathrm{IC}_U=\mathbf{R}g_!(\mathrm{IC}_{\mathscr{U}} \otimes^{\mathbf{L}} \mathbf{R}f_!\mathbf{K})[1]=(\mathbf{R}g_!\mathrm{IC}_{\mathscr{U}} \otimes^{\mathbf{L}} \mathbf{R}\lambda_!\mathbf{K})[1]
\]
and hence for $i \leq 1-n$ we have (using \eqref{1150-4} and \eqref{1150-5})
\[
\mathbf{R}^i\lambda_!\mathrm{IC}_U=\mathbf{R}^{1-n}g_!\mathrm{IC}_{\mathscr{U}} \otimes \mathbf{R}^{i+n}\lambda_!\mathbf{K}=\begin{cases}
0 & \text{if } i<1-n \\
\mathbf{K} & \text{if } i=1-n
\end{cases}
\]
Then $\mathrm{IH}_c^1(U)=\mathbf{H}_c^{1-n}(U,\mathrm{IC}_U)=H^0_c(U/\mathbf{G}_m,\mathbf{R}^{1-n}\lambda_!\mathrm{IC}_U)=H^0_c(U/\mathbf{G}_m,\mathbf{K})$.
\end{proof}
\section{Proof of Main Theorem}\label{iff-part}
\subsection{Surjectivity}
In this subsection we describe $\mathbf{G}_m$-invariant open subsets with proper good quotients, using the characterizations established in \S\ref{Geom-char}.
\subsubsection{Exceptional case}
First we deal with the exceptional case, as identified before, that the open subset intersects the source or the sink.
\begin{prop}[\cite{MR755486}, Proposition 2.4 if $k=\mathbf{C}$]\label{prop:typei}
Let $X$ be a normal proper variety over a field $k$ with a $\mathbf{G}_m$-action. Let $U \subseteq X$ be a $\mathbf{G}_m$-invariant open subset with a proper good quotient. If $U$ intersects the source $X_{\min}$ (resp., the sink $X_{\max}$) of $X$, then $U=X_{\min}^-$ (resp., $U=X_{\max}^+$).
\end{prop}
\begin{proof}
The quotient stack $[U/\mathbf{G}_m]$ satisfies the conditions (E) and ($\Theta$) by Theorem \ref{thm0}. If $U \cap X_{\min} \neq \emptyset$ (resp., $U \cap X_{\max} \neq \emptyset$), then $X_{\min} \subseteq U$ (resp., $X_{\max} \subseteq U$) by Lemma \ref{A^0}. In this case the evaluation $\mathrm{ev}_1: X_{\min}^- \to U$ (resp., $\mathrm{ev}_1: X_{\max}^+ \to U$) is both an open (Lemma \ref{source-sink}) and a closed immersion (Lemma \ref{theta-red}), we are done.
\end{proof}
\subsubsection{General case}
If the open subset does not intersect the source or the sink, then again we find that its complement is disconnected.
\begin{prop}\label{2028-3}
Let $X$ be a normal proper variety over a field $k$ with a $\mathbf{G}_m$-action. Let $U \subseteq X$ be a $\mathbf{G}_m$-invariant open dense subset with a proper good quotient. The following are equivalent:
\begin{enumerate}
\item
The subset $U$ does not intersect the source $X_{\min}$ or the sink $X_{\max}$ of $X$.
\item
The complement $Z:=X-U$ has two connected components.
\end{enumerate}
In this case, $X_{\min}$ and $X_{\max}$ are in the different connected components of $Z$.
\end{prop}
\begin{proof}
Our arguments are similar to those in \cite[Theorem 2.11]{MR755486}. 

The implication $(2) \Rightarrow (1)$. If $U \cap X_{\min} \neq \emptyset$ (resp., $U \cap X_{\max} \neq \emptyset$), then $U=X_{\min}^-$ (resp., $U=X_{\max}^+$) by Proposition \ref{prop:typei}. In this case $Z$ is connected since by Theorem \ref{summary} every point in $Z$ can be connected to $X_{\max}$ (resp., $X_{\min}$) via any smoothable maximal chain of orbits in $X$ passing through it.

The implication $(1) \Rightarrow (2)$. Consider the $\mathbf{G}_m$-invariant fixed-point-free subset
\[
U^\circ:=U-\bigsqcup_{i \in \pi_0(X^0)} (U \cap X_i)^- \subseteq U.
\]
Since the quotient stack $[U/\mathbf{G}_m]$ is $\Theta$-reductive, the subset $U^\circ \subseteq U$ is open by Lemma \ref{theta-red}. To conclude it suffices to show the quotient stack $[U^\circ/\mathbf{G}_m]$ admits a proper good moduli space. Indeed, if $[U^\circ/\mathbf{G}_m]$ admits a proper good moduli space, by Proposition \ref{prop:2-connected} the complement $X-U^\circ$ has two connected components, one containing $X_{\min}$ and the other containing $X_{\max}$. Then $Z$ is disconnected as $X_{\min}$ and $X_{\max}$ are still in the different connected components. We are done by Corollary \ref{1727}.

To show the quotient stack $[U^\circ/\mathbf{G}_m]$ admits a proper good moduli space we apply Theorem \ref{thm0}. It is $\Theta$-reductive by Lemma \ref{theta-red} since it has no fixed point. For any smoothable maximal chain of orbits $f_{\mathbf{K}}: C_{\mathbf{K}} \to X$ in $X$, we have 
\[
f_{\mathbf{K}}(C_{\mathbf{K}}) \cap U^\circ \subseteq f_{\mathbf{K}}(C_{\mathbf{K}}) \cap U.
\]
By Theorem \ref{summary} there are two cases: If $f_{\mathbf{K}}(C_{\mathbf{K}}) \cap U=\mathbf{G}_m.x$ for some $x \in X-X^0$, then $f_{\mathbf{K}}(C_{\mathbf{K}}) \cap U^\circ=\mathbf{G}_m.x$. If $f_{\mathbf{K}}(C_{\mathbf{K}}) \cap U=\mathbf{G}_m.x_1 \cup \{x_1^+=x_2^-\} \cup \mathbf{G}_m.x_2$ for some $x_1,x_2 \in X-X^0$, then $f_{\mathbf{K}}(C_{\mathbf{K}}) \cap U^\circ=\mathbf{G}_m.x_1$. This shows that $[U^\circ/\mathbf{G}_m]$ satisfies the conditions (S) and (E) by Theorem \ref{summary}.
\end{proof}
Hereafter, let $U \subseteq X$ be a $\mathbf{G}_m$-invariant open dense subset with a proper good quotient and denote by $Z:=X-U$ its complement. Suppose $U$ does not intersect the source or the sink of $X$. Let $\Omega_1 \subseteq Z$ (resp., $\Omega_2 \subseteq Z$) be the connected component containing $X_{\min}$ (resp., $X_{\max}$). By Remark \ref{new-2} we have
\[
U \subseteq \bigsqcup_{U \cap X_i^- \cap X_j^+ \neq \emptyset} X_i^- \cap X_j^+ \subseteq \bigsqcup_{i \in \Delta^- \cup \Delta^0 \atop j \in \Delta^0 \cup \Delta^+} X_i^- \cap X_j^+,
\]
where 
\[
\Delta^\pm:=\{i: U \cap (X_i^\pm-X_i) \neq \emptyset\} \text{ and } \Delta^0:=\{i: U \cap X_i \neq \emptyset\}.
\]
By Theorem \ref{summary} fixed point components in $\Delta^-$ (resp., $\Delta^+$) can connect to $X_{\min}$ (resp., $X_{\max}$) inside $Z$, i.e., they are in the same connected component of $Z$ as $X_{\min}$ (resp., $X_{\max}$). Let us define
\begin{align*}
\Delta^- \subseteq A^-:&=\{i: X_i \subseteq \Omega_1\}=\{i: \Omega_1 \cap X_i \neq \emptyset\}, \\
A^0:&=\Delta^0=\{i: X_i \subseteq U\}=\{i: U \cap X_i \neq \emptyset\}, \\
\Delta^+ \subseteq A^+:&=\{i: X_i \subseteq \Omega_2\}=\{i: \Omega_2 \cap X_i \neq \emptyset\},
\end{align*}
where the equality hold by Lemma \ref{A^0}. The triple $(A^-,A^0,A^+)$ is a non-trivial division of $\pi_0(X^0)$. Moreover
\[
U \subseteq \bigsqcup_{i \in A^- \cup A^0 \atop j \in A^0 \cup A^+} X_i^- \cap X_j^+.
\]
Surprisingly, this inclusion turns out to be an equality so the open subset $U$ can be recovered from this triple $(A^-,A^0,A^+)$.
\begin{thm}[\cite{MR755486}, Theorem 2.11 if $k=\mathbf{C}$]
Let $X$ be a normal proper variety over a field $k$ with a $\mathbf{G}_m$-action. Let $U \subseteq X$ be a $\mathbf{G}_m$-invariant open subset with a proper good quotient. If $U$ does not intersect the source or the sink, then
\[
U=\bigsqcup_{i \in A^- \cup A^0 \atop j \in A^0 \cup A^+} X^-_i \cap X_j^+.
\]
\end{thm}
\begin{proof}
For any point $x \in U$, say $x \in X_i^- \cap X_j^+$ for some $i,j \in \pi_0(X^0)$, we show $i \in A^- \cup A^0 \text{ and } j \in A^0 \cup A^+$. For this we choose a smoothable maximal chain of orbits $f_{\mathbf{K}}: C_{\mathbf{K}} \to X$ in $X$ passing through $x$. If $j \in A^-$ (resp., $i \in A^+$), then by Theorem \ref{summary} we have
\[
\mathbf{G}_m.x \cup \{x^+\} \subseteq U \cap f_{\mathbf{K}}(C_{\mathbf{K}}) \text{ (resp., } \{x^-\} \cup \mathbf{G}_m.x \subseteq U \cap f_{\mathbf{K}}(C_{\mathbf{K}})).
\]
In particular $U \cap X_j \neq \emptyset$ (resp., $U \cap X_i \neq \emptyset$), i.e., $j \in A^0$ (resp., $i \in A^0$), a contradiction since $(A^-,A^0,A^+)$ is a division of $\pi_0(X^0)$.

Conversely, for any point $x \in X_i^- \cap X_j^+$ with $i \in A^- \cup A^0$ and $j \in A^0 \cup A^+$, we show $x \in U$. If $x \in \Omega_1$ (resp., $x \in \Omega_2$), then its complete orbit map $\overline{\sigma}_x$ locates in $\Omega_1$ (resp., $\Omega_2$) as $\Omega_1 \subseteq X$ (resp., $\Omega_2 \subseteq X$) is closed. In particular $x^+ \in X_j \cap \Omega_1$ (resp., $x^- \in X_i \cap \Omega_2$), i.e., $j \in A^-$ (resp., $i \in A^+$), a contradiction.
\end{proof}
The triple $(A^-,A^0,A^+)$ extracted from $U$ is thus of great importance and we derive some of its properties. Corollary \ref{hhh} tells that $X_j <_d X_i$ for some $i \in A^0$ implies $j \in A^-$, which can be generalized as follows since $A^\pm$ can be seen as the saturation of $\Delta^\pm$ with respect to the relation $<$ (resp., $>$).
\begin{lem}
If $i \in A^- \cup A^0$ and  $X_j<X_i$, then $j \in A^-$.
\end{lem}
\begin{proof}
By induction we may assume that $X_j <_d X_i$. Choose a smoothable maximal chain of orbits $f_{\mathbf{K}}: C_{\mathbf{K}} \to X$ in $X$ passing through $X_j^- \cap X_i^+$, then $U$ cannot intersect $f_{\mathbf{K}}(C_{\mathbf{K}})$ before $X_i$ (along the chain $f_{\mathbf{K}}$), otherwise $X_i$ would be in the same connected component of $Z$ as $X_{\max}$ by Theorem \ref{summary}, i.e., $i \in A^+$, a contradiction. This implies that $X_j \cap Z \neq \emptyset$ and hence $X_j$ is in the same connected component of $Z$ as $X_{\min}$, i.e., $j \in A^-$.
\end{proof}
These properties make up the notion of semi-sections. 
\begin{defn}[\cite{MR704983}, Definition 1.2 and 1.3]\label{def:ss}
Let $X$ be a variety over a field $k$ with a $\mathbf{G}_m$-action. A \emph{semi-section} on $X$ is a non-trivial division of $\pi_0(X^0)$ into three subsets $(A^-,A^0,A^+)$, satisfying one of the following equivalent conditions:
\begin{itemize}
\item
If $i \in A^- \cup A^0$ and $X_j<X_i$, then $j \in A^-$.
\item
If $i \in A^0 \cup A^+$ and $X_i<X_j$, then $j \in A^+$. 
\end{itemize}
It is a \emph{section} on $X$ if $A^0=\emptyset$. If $(A^-,A^0,A^+)$ is a semi-section on $X$, then
\[
X(A^\ast):=\bigsqcup_{i \in A^- \cup A^0 \atop j \in A^0 \cup A^+} X^-_i \cap X_j^+ \subseteq X
\]
is the \emph{semi-sectional subset} defined by the semi-section $(A^-,A^0,A^+)$.
\end{defn}
\begin{rmk}\label{f0}
Let $(A^-,A^0,A^+)$ be a semi-section on $X$. By definition
\begin{enumerate}
\item 
No fixed point components in $A^0$ are comparable.
\item 
It follows that $i \in A^0$ if and only if $X(A^*) \cap X_i \neq \emptyset$, then $X_i \subseteq X(A^*)$.
\item 
If $X_i<X_j$ and $X_j<X_i$ for some $i \neq j$, then $i,j \in A^-$ or $i,j \in A^+$.
\item 
If $X$ is proper, then $\min \in A^- \cup A^0$ and $\max \in A^0 \cup A^+$ by Corollary \ref{well-order}. In particular any semi-sectional subset is dense since it contains the dense subset $X_{\min}^- \cap X_{\max}^+$ (see Lemma \ref{source-sink}). So by (3) the first necessary condition for the existence of semi-sections on $X$ is $X_{\max} \nless X_{\min}$. This holds, e.g. if $X$ is normal (see Proposition \ref{q'-cycle} later).
\end{enumerate}
\end{rmk}
This definition also covers the exceptional case Proposition \ref{prop:typei} as $X_{\min}^-$ and $X_{\max}^+$ are semi-sectional subsets corresponding to the semi-sections
\[
(\emptyset,\{\min\},\pi_0(X^0)-\{\min\}) \text{ and } (\pi_0(X^0)-\{\max\},\{\max\},\emptyset) \text{ respectively.}
\]
Indeed, this amounts to saying that there is no fixed point component lying outside the source or the sink, which is guaranteed by the following result.
\begin{prop}[\cite{MR709583}, Corollary A.3 if $k=\mathbf{C}$]\label{q'-cycle}
Let $X$ be a normal proper variety over a field $k$ with a $\mathbf{G}_m$-action. Then $X_i^+=X_i$ (resp., $X_i^-=X_i$) if and only if $i=\min$ (resp., $i=\max$).
\end{prop}
\begin{rmk}
This gives a geometric characterization of the source (i.e., there is no orbit flowing in) and the sink (i.e., there is no orbit flowing out). If $X$ is smooth, this claim follows directly from a dimension computation on tangent spaces as in \cite[Proposition 1.4.11 (vi)]{Drinfeld-Gm}.
\end{rmk}
\begin{proof}
We only prove the plus case. If $i \neq \min$, then $X_{\min}<X_i$ by Corollary \ref{well-order} and hence $X_i \subsetneq X_i^+$. It remains to show $X_{\min}^+ \subseteq X_{\min}$. If there exists a point $x \in X_{\min}^+-X_{\min}$, we choose a $\mathbf{G}_m$-invariant affine neighbourhood $U \subseteq X$ of $x^+ \in X_{\min}$. Then the closed immersion (\cite[Proposition 1.4.11 (iv)]{Drinfeld-Gm}) $\mathrm{ev}_1: U^- \to U$ is also open as $\mathrm{ev}_1: (U \cap X_{\min})^- \to X$ is, so $U^-=U$. By the computation in \cite[Proposition 1.4.11 (vi)]{Drinfeld-Gm}
\begin{enumerate}
\item
the morphism $\mathrm{ev}_1: U^- \to U$ induces an isomorphism
\[
\mathbf{T}_{x^+}(X)=\mathbf{T}_{x^+}(U)=\mathbf{T}_{x^+}(U^-) \xrightarrow{\sim} \mathbf{T}_{x^+}(U)^+=\mathbf{T}_{x^+}(X)^+,
\]
i.e., the tangent space $\mathbf{T}_{x^+}(X)$ has only non-negative weights part.
\item
the morphism $\mathrm{ev}_1: U^+ \to U$ induces an isomorphism
\[
0 \neq \mathbf{T}_{x^+}(U^+) \xrightarrow{\sim} \mathbf{T}_{x^+}(U)^-=\mathbf{T}_{x^+}(X)^-,
\]
contributing to the negative weights part of $\mathbf{T}_{x^+}(X)$ as $x \notin X_{\min}$.
\end{enumerate}
a contradiction.
\end{proof}
In summary we are able to reformulate our conclusions uniformly.
\begin{thm}[\cite{MR709583}, Theorem 1.4 and \cite{MR755486}, Theorem 2.11 if $k=\mathbf{C}$]\label{onlyif}
Let $X$ be a normal proper variety over a field $k$ with a $\mathbf{G}_m$-action. If $U \subseteq X$ is a $\mathbf{G}_m$-invariant open dense subset with a proper good quotient, then $U=X(A^\ast)$ for some semi-section $(A^-,A^0,A^+)$ on $X$.
\end{thm}
As a consequence, there are at most $2+3^{|\pi_0(X^0)|-2}$ $\mathbf{G}_m$-invariant open subsets of $X$ with a proper good quotient. The bound is sharp due to the examples of the standard $\mathbf{G}_m$-action on $\mathbf{P}^1$ or \cite[\S 3]{MR755486}.
\subsection{Injectivity}
In this subsection, we show that semi-sectional subsets are $\mathbf{G}_m$-invariant open dense and admit proper good quotients. This gives the correspondence in \textbf{Main Theorem} and then we show it is injective.
\begin{prop}[\cite{MR704983}, Proposition 2.4 if $k=\bar{k}$]
Let $X$ be a proper variety over a field $k$ with a $\mathbf{G}_m$-action. The semi-sectional subsets of $X$ are $\mathbf{G}_m$-invariant and open dense.
\end{prop}
\begin{proof}
For any semi-section $(A^-,A^0,A^+)$ on $X$, the semi-sectional subset $X(A^\ast) \subseteq X$ is $\mathbf{G}_m$-invariant and dense by Remark \ref{f0} (4). To show openness, note that
\[
X(A^\ast)= \bigsqcup_{i \in A^- \cup A^0} X^-_i \quad \cap \bigsqcup_{j \in A^0 \cup A^+} X_j^+
\]
is an intersection of two open subsets of $X$ by the following lemma.
\end{proof}
\begin{lem}\label{1544}
Let $X$ be a proper variety over a field $k$ with a $\mathbf{G}_m$-action. For any subset $\Delta \subseteq \pi_0(X^0)$, we have
\[
\bigsqcup_{i \in \Delta} X_i^- \subseteq X \text{ is closed if and only if } \Delta \text{ is saturated with respect to } >.
\]
Similar result holds for $+$ case.
\end{lem}
\begin{proof}
This can be seen as a slight generalization of \cite[Lemma 1.3.1]{MR709583}.
\begin{enumerate}
\item
\textsc{If Part}: For any point $x \in \overline{\bigsqcup_{i \in \Delta} X_i^-}=\bigsqcup_{i \in \Delta} \overline{X_i^-} \subseteq X$, say $x \in \overline{X_\ell^-}$ for some $\ell \in \Delta$, we show $x \in X_j^-$ for some $j \in \Delta$. Passing to an irreducible component of $X_\ell^-$ whose closure contains $x$, we may assume that $X_\ell^-$ is irreducible. Consider $x^- \in \overline{X_\ell^-} \cap X^0=\overline{X_\ell^-}^0$. If $\Omega$ is the connected component of $\overline{X_\ell^-}^0$ containing $x^-$ and assume $\Omega \subseteq X_j$ for some $j \in \pi_0(X^0)$, then $x \in X_j^-$. Since $X_\ell^-$ is dense in $\overline{X_\ell^-}$, we have that $X_\ell$ is the source of $\overline{X_\ell^-}$ by Lemma \ref{source-sink} and hence $X_\ell < X_j$ by Corollary \ref{well-order}. Then $j \in \Delta$ as $\Delta$ is saturated with respect to $>$.
\item
\textsc{Only If Part}: We show $j \in \Delta$ for any $X_j> X_\ell$ with $\ell \in \Delta$. We may assume that $X_j>_d X_\ell$. In this case $X_\ell^- \cap X_j^+ \neq \emptyset$ and hence $X_j \cap \overline{X_\ell^-} \neq \emptyset$. The closedness of $\bigsqcup_{i \in \Delta} X_i^- \subseteq X$ implies that
\[
X_j \cap \bigsqcup_{i \in \Delta} X_i^-=X_j \cap \overline{\bigsqcup_{i \in \Delta} X_i^-}=X_j \cap \bigsqcup_{i \in \Delta} \overline{X_i^-}=\bigsqcup_{i \in \Delta} X_j \cap \overline{X_i^-} \neq \emptyset,
\]
i.e., $X_j \cap X_s^- \neq \emptyset$ for some $s \in \Delta$. This implies that $j=s \in \Delta$.
\end{enumerate}
\end{proof}

\begin{thm}[\cite{MR709583}, Theorem 1.3 and \cite{MR755486}, Theorem 2.8 if $k=\mathbf{C}$]\label{prop:if}
Let $X$ be a normal proper variety over a field $k$ with a $\mathbf{G}_m$-action. For any semi-section $(A^-,A^0,A^+)$ on $X$, the semi-sectional subset $X(A^\ast) \subseteq X$ is a $\mathbf{G}_m$-invariant open dense subset with a proper good quotient.
\end{thm}
\begin{proof}
It is equivalent to show the quotient stack $[X(A^*)/\mathbf{G}_m]$ admits a proper good moduli space. For this we check the conditions in Theorem \ref{thm0}.
\begin{enumerate}
\item[($\Theta$)]
By Lemma \ref{theta-red}, we show the evaluation $\mathrm{ev}_1: (X(A^*) \cap X_i)^\pm \to X(A^*)$ is a closed immersion for each $i$. 

Suppose $X(A^*) \cap X_i \neq \emptyset$. Then $i \in A^0$ and $X_i \subseteq X(A^*)$ by Remark \ref{f0} (2). Then we reduce to show $X_i^\pm \subseteq X(A^*)$ is closed for each $i \in A^0$. For any point $x \in \overline{X_i^-}^{X(A^*)} \subseteq X(A^*)$, as in the proof of Lemma \ref{1544} we have that $x \in X_j^-$ for some $X_i<X_j$. Since $x \in X(A^*)$ this means that $j \in A^- \cup A^0$ and then $i \in A^0$ forces that $i=j$, as desired.
\item[(S)]
By Proposition \ref{f2}, we show for any smoothable maximal chain of orbits $f_\mathbf{K}: C_\mathbf{K} \to X$ in $X$, if $\mathbf{G}_m.x_1 \neq \mathbf{G}_m.x_2 \subseteq X(A^*) \cap f_\mathbf{K}(C_\mathbf{K})$, then we have either $x_1^+=x_2^- \in X(A^*)$ or $x_1^-=x_2^+ \in X(A^*)$.

Suppose $x_1$ appears earlier than $x_2$ along the chain $f_\mathbf{K}$, i.e., if $x_1^+ \in X_i$ and $x_2^- \in X_j$, then $X_i<X_j$. Since $x_1,x_2 \in X(A^*)$ this means $i \in A^0 \cup A^+$ and $j \in A^- \cup A^0$, then $i=j$, as desired.
\item[(E)]
By Proposition \ref{f1}, we show for any smoothable maximal chain of orbits $f_\mathbf{K}: C_\mathbf{K} \to X$ in $X$, we have $X(A^*) \cap f_\mathbf{K}(C_\mathbf{K}) \neq \emptyset$. 

If $X(A^*) \cap f_\mathbf{K}(C_\mathbf{K}) = \emptyset$, then
\[
\{\min,\max\} \subseteq \{i: X_i \cap f_\mathbf{K}(C_\mathbf{K}) \neq \emptyset\} \subseteq A^- \text{ or } A^+.
\]
The first inclusion follows since $f_\mathbf{K}$ is maximal. Suppose the second inclusion fails, then we can find a point $x \in f_\mathbf{K}(C_\mathbf{K})$ such that $x \in X_i^- \cap X_j^+$ for some $i \in A^- \cup A^0$ and $j \in A^0 \cup A^+$. By definition $x \in X(A^*)$, i.e., $x \in X(A^*) \cap f_\mathbf{K}(C_\mathbf{K}) \neq \emptyset$, a contradiction. But then $\{\min,\max\} \subseteq A^-$ or $A^+$ implies that $A^0 \cup A^+=\emptyset$ or $A^- \cup A^0=\emptyset$ by Corollary \ref{well-order}, again a contradiction.
\end{enumerate}
\end{proof}
Therefore we have a correspondence
\begin{align*}
\{ 
\text{Semi-sections on } X
\}
&\to 
\begin{Bmatrix} 
\mathbf{G}_m\text{-invariant open dense subsets of } X  \\
\text{with proper good quotients}
\end{Bmatrix} \\
(A^-,A^0,A^+) &\mapsto X(A^\ast),
\end{align*}
by mapping each semi-section to its semi-sectional subset. Under this correspondence, sections on $X$ correspond to open subsets with proper geometric quotients by Remark \ref{f0} (2) and the fact that a good quotient is geometric if and only if it contains no fixed point. To conclude we show the correspondence is injective.
\begin{lem}
Let $X$ be a proper variety over a field $k$ with a $\mathbf{G}_m$-action. Different semi-sections on $X$ define different semi-sectional subsets of $X$.
\end{lem}
\begin{proof}
If $(A^-,A^0,A^+) \neq (B^-,B^0,B^+)$ are semi-sections on $X$, then
\[
A^- \cup A^0 \neq B^- \cup B^0 \text{ or } A^0 \cup A^+ \neq B^0 \cup B^+.
\]
It suffices to consider the first case. We may assume that $(A^- \cup A^0) \cap B^+ \neq \emptyset$. For any $i \in (A^- \cup A^0) \cap B^+$, by Corollary \ref{well-order} there exists $j \in A^0 \cup A^+$ such that $X_i^- \cap X_j^+ \neq \emptyset$, i.e., $X_i<_d X_j$. Then $X_i^- \cap X_j^+ \subseteq X(A^\ast)$ and $j \in B^+$, so $X_i^- \cap X_j^+ \nsubseteq X(B^\ast)$ and this tells the difference between $X(A^\ast)$ and $X(B^\ast)$.
\end{proof}
As an application we show that any normal projective $\mathbf{G}_m$-variety is covered by open subsets with proper good quotients.
\begin{lem}\label{q-cycle}
Let $X$ be a normal quasi-projective variety over a field $k$ with a $\mathbf{G}_m$-action. Then $X_i<X_j$ and $X_j<X_i$ imply that $i=j$.
\end{lem}
The assertion could fail for normal proper varieties, see the examples in \cite[\S 1]{MR704991}.
\begin{proof}
Since $X$ is normal and quasi-projective, by \cite[Corollary 2.14]{MR3738079} there exists an ample $\mathbf{G}_m$-linearied line bundle $\mathscr{L}$ over $X$. For any fixed point $x \in X^0$ we denote by $\mathrm{wt}_{\mathbf{G}_m}(\mathscr{L},x) \in X^*(\mathbf{G}_m) \cong \mathbf{Z}$ the character of $\mathbf{G}_m$ given by the $\mathbf{G}_m$-action on the fiber $\mathscr{L}_x$ of $\mathscr{L}$ at $x$. The character is constant on each connected component of $X^0$ and thus defines a map
\[
\mathrm{wt}_{\mathscr{L}}: \pi_0(X^0) \to \mathbf{Z} \text{ mapping } i \mapsto \mathrm{wt}_{\mathbf{G}_m}(\mathscr{L},x) \text{ for some } x \in X_i.
\]
To conclude, we claim that $X_i<X_j$ for $i \neq j$ implies $\mathrm{wt}_{\mathscr{L}}(i)<\mathrm{wt}_{\mathscr{L}}(j)$. We may assume that $X_i <_d X_j$ and in this case there exists a non-constant complete orbit map $\overline{\sigma}_x: \mathbf{P}^1_{\mathbf{K}}=\mathrm{Proj}(\mathbf{K}[x_0,x_1]) \to X$ for some point $x \in X_i^- \cap X_j^+(\mathbf{K})$, such that the homogeneous coordinates $x_0$ and $x_1$ have $\mathbf{G}_m$-weight $0$ and $1$ respectively. Since $\mathscr{L}$ is ample, by \cite[Remark 2.2]{MR3758902} we have
\[
0<\deg(\mathscr{L}|_{\mathbf{P}^1_{\mathbf{K}}})=\dfrac{\mathrm{wt}_{\mathbf{G}_m}(\mathscr{L},x^+)-\mathrm{wt}_{\mathbf{G}_m}(\mathscr{L},x^-)}{\mathrm{wt}_{\mathbf{G}_m}(x_1)-\mathrm{wt}_{\mathbf{G}_m}(x_0)}=\mathrm{wt}_{\mathscr{L}}(j)-\mathrm{wt}_{\mathscr{L}}(i).
\]
\end{proof}
\begin{prop}[\cite{MR709583}, Theorem 1.6 if $k=\mathbf{C}$ and $X^0=\emptyset$]\label{prop-app}
Let $X$ be a normal projective variety over a field $k$ with a $\mathbf{G}_m$-action. Then $X$ is covered by $\mathbf{G}_m$-invariant open subsets with proper good quotients.
\end{prop}
\begin{proof}
By Theorem \ref{prop:if} it suffices to show any point of $X$ is contained in a semi-sectional subset, i.e., for any point $x \in X$, say $x \in X_i^- \cap X_j^+$ for some $i,j \in \pi_0(X^0)$, there exists a semi-section $(A^-,A^0,A^+)$ on $X$ such that $i \in A^- \cup A^0$ and $j \in A^0 \cup A^+$. To start we define
\begin{gather*}
\Delta^-:=\{\ell: X_\ell < X_i\} \text{ and } \Delta^+:=\{\ell: X_\ell > X_j\}.
\end{gather*}
For any $p \in \Delta^-$ and $q \in \Delta^+$ with $p \neq q$, by definition $X_q > X_j > X_i > X_p$ and hence $X_p \ngtr X_q$ by Lemma \ref{q-cycle}. Let $(A^-,A^+)$ be a maximal pair of subsets of $\pi_0(X^0)$ such that
\begin{enumerate}
\item $\Delta^\pm \subseteq A^\pm$ and
\item $X_p \ngtr X_q$ for any $p \in A^-$ and $q \in A^+$ with $p \neq q$.
\end{enumerate}
Then $i \in A^--A^+$ and $j \in A^+-A^-$. Let $A^0:=A^- \cap A^+$ and we claim that
\[
\text{The triple } (A^--A^0,A^0,A^+-A^0) \text{ is a semi-section on } X.
\]
It remains to show $A^- \cup A^+=\pi_0(X^0)$. Indeed, for any $\ell \in \pi_0(X^0)-A^--A^+$, the maximality of the pair $(A^-,A^+)$ implies that there exist $p \in A^-$ and $q \in A^+$ such that $X_p>X_\ell>X_q$, then necessarily $p=q$ and hence $\ell=p$ by Lemma \ref{q-cycle}, a contradiction.
\end{proof}

\end{document}